\documentclass[journal]{IEEEtran} 
\usepackage{cite}
\usepackage{amsmath,amssymb,amsfonts}
\usepackage{amsthm}

\usepackage{algorithmic}
\usepackage{graphicx}
\usepackage{textcomp}

\usepackage{url}
\usepackage{xcolor,hyperref}
\definecolor{darkblue}{rgb}{0,0,1}
\hypersetup{colorlinks,breaklinks,
    linkcolor=darkblue,urlcolor=darkblue,anchorcolor=darkblue,citecolor=darkblue}

\usepackage{tikz}
\usetikzlibrary{arrows.meta,positioning}
\usepackage{algorithm}
\usepackage{algorithmic} 
\usepackage{float}
\usepackage{float}
\usepackage{subfig}
\usepackage{caption}
\usepackage{tabularx}
\usepackage[flushleft]{threeparttable} 

\usepackage{enumitem}
\usepackage{makecell} 
\usepackage{booktabs}
\usepackage{nicefrac}
\usepackage{pifont}
\usepackage{comment}
\newtheorem{theorem}{Theorem}
\newtheorem{lemma}{Lemma}
\newtheorem{assumption}{Assumption}

\newtheorem{remark}{Remark}
\newtheorem{corollary}{Corollary}

\newcommand\prox[1]{\mathrm{prox}_{#1}}
\newcommand\proxp[1]{\mathrm{prox}_{\gamma\varphi}\left(#1\right)}
\newcommand\proxpi[1]{\mathrm{prox}_{\gamma\varphi}(#1)}
\newcommand\Fnor[1]{F^{\gamma}_{\mathrm{nor}}(#1)}

\newcommand{\order}[1]{\mathcal{O}\left(#1\right)}

\newcommand{\prt}[1]{\left(#1\right)}
\newcommand{\brk}[1]{\left[#1\right]}
\newcommand{\crk}[1]{\left\{#1\right\}}

\newcommand{\orderi}[1]{\mathcal{O}(#1)}
\newcommand{\prti}[1]{(#1)}
\newcommand{\brki}[1]{[#1]}
\newcommand{\crki}[1]{\{#1\}}

\newcommand{\inpro}[1]{\left\langle #1 \right\rangle}
\newcommand{\inproi}[1]{\langle #1\rangle}
\newcommand{\condE}[2]{\E\brk{#1\middle|#2}}
\newcommand{\condEi}[2]{\E[#1|#2]}
\newcommand{\norm}[1]{\left\Vert #1 \right\Vert}
\newcommand{\normi}[1]{\Vert #1 \Vert}
\newcommand{\bs}[1]{\boldsymbol{#1}}
\newcommand{\hLambda}{\hat{\Lambda}}

\newcommand{\R}{\mathbb{R}}
\newcommand{\E}{\mathbb{E}}

\newcommand{\x}{\mathbf{x}}
\newcommand{\z}{\mathbf{z}}
\newcommand{\g}{\mathbf{g}}
\newcommand{\1}{\mathbf{1}}

\newcommand{\T}{\intercal}
\newcommand{\sumn}{\sum_{i=1}^n}
\newcommand{\cN}{\mathcal{N}}
\renewcommand{\d}{\mathbf{d}}
\newcommand{\y}{\mathbf{y}}
\newcommand{\s}{\mathbf{s}}

\newcommand{\hU}{\hat{U}}

\newcommand{\cF}{\mathcal{F}}

\newcommand{\cL}{\mathcal{L}}
\newcommand{\0}{\mathbf{0}}

\newcommand{\uf}{f^*}

\newcommand{\sigfn}{\sigma_f^*}
\newcommand{\cC}{\mathcal{C}}
\newcommand{\cH}{\mathcal{H}}

\newcommand{\cA}{\mathcal{A}}
\newcommand{\KTN}{K_{\text{Transient}}^{(\text{NCVX})}}

\newcommand{\Ln}{L_{\mathrm{F}}}
\newcommand{\Fnorb}[1]{\boldsymbol{F}^{\gamma}_{\mathrm{nor}}(#1)}
\newcommand{\Fnori}[2]{{F}^{\gamma}_{#1,\mathrm{nor}}(#2)}

\newcommand{\barz}{\1\bar{z}^{\T}}
\newcommand{\ux}{\underline{x}}
\newcommand{\bpsi}{\bar{\psi}}
\newcommand{\cM}{\mathcal{M}}

\newcommand{\C}{C_0}

\newcommand{\tk}{\boldsymbol{k}}

\newcommand{\bproxp}[1]{\boldsymbol{\mathrm{prox}}_{\gamma\varphi}\left(#1\right)}
\newcommand{\bu}{\mathbf{u}}

\newcommand{\ceili}[1]{\lceil#1\rceil}
\newcommand{\normgt}{norM-DSGT}
\newcommand{\normed}{norM-ED}
\newcommand{\LA}{L_{\mathrm{nat}}}
\newcommand{\U}{\mathcal{U}}
\newcommand\dom[1]{\mathrm{dom}(#1)}
\newcommand\Fnat[1]{F^{\gamma}_{\mathrm{nat}}(#1)}
\newcommand\diag{\mathrm{diag}}
\newcommand{\cS}{\mathcal{S}}

\newcommand{\ccD}{\mathcal{T}}

\newcommand{\tc}{\chi}

\definecolor{cuhkpl}{RGB}{152,24,147} 	

\def\BibTeX{{\rm B\kern-.05em{\sc i\kern-.025em b}\kern-.08em
    T\kern-.1667em\lower.7ex\hbox{E}\kern-.125emX}}
\begin{document}
\title{Distributed Normal Map-based Stochastic Proximal Gradient Methods over Networks}
\author{Kun Huang, Shi Pu, and Angelia Nedi\'{c}
\thanks{This work was partially supported by the National Natural Science Foundation of China (NSFC) under Grant 62373316 and Grant 62336005. Kun Huang and Shi Pu are with the School of Data Science, The Chinese University of Hong Kong, Shenzhen 518172, China. Angelia Nedi\'{c} is with the School of Electrical, Computer and Energy Engineering, Arizona State University, Tempe, AZ, 85281, USA (e-mail: kunhuang@link.cuhk.edu.cn; pushi@cuhk.edu.cn; angelia.nedich@asu.edu).
}}

\maketitle

\begin{abstract}

Consider $n$ agents connected over a network collaborating to minimize the average of their local cost functions combined with a common nonsmooth function.
This paper introduces a unified algorithmic framework for solving such a problem through distributed stochastic proximal gradient methods, leveraging the normal map update scheme. 
Within this framework, we propose two new algorithms, termed Normal Map-based Distributed Stochastic Gradient Tracking (norM-DSGT) and Normal Map-based Exact Diffusion (norM-ED).
We demonstrate that both methods can asymptotically achieve comparable convergence rates to the centralized stochastic proximal gradient descent method under a general variance condition on the stochastic gradients. Additionally, the number of iterations required for norM-ED to achieve such a rate (i.e., the transient time) behaves as $\mathcal{O}(n^{3}/(1-\lambda)^2)$ for minimizing composite objective functions, matching the performance of the non-proximal ED algorithm. Here $1-\lambda$ denotes the spectral gap of the mixing matrix related to the underlying network topology. To our knowledge, such a convergence result is state-of-the-art for the considered composite problem. Under the same condition, norM-DSGT enjoys a transient time of {$\mathcal{O}(\max\{n^3/(1-\lambda)^2, n/(1-\lambda)^3\})$, { which matches that of the non-proximal DSGT algorithm and norM-ED under the condition $(1-\lambda)^{-1}=\mathcal{O}(n^{2})$}.}

\end{abstract}


\section{Introduction}
\label{sec:introduction}
\IEEEPARstart{I}{n} this paper, we investigate how a group of networked agents $\cN:=\crki{1,2,\ldots, n}$ collaborate to solve the following distributed composite optimization problem:
\begin{equation}
    \label{eq:P}
    \min_{x\in \R^p} \psi(x):= f(x) + \varphi(x),\; f(x) := \frac{1}{n}\sumn f_i(x),
\end{equation}
where each agent $i$ only has access to its local objective function $f_i:\R^p\rightarrow\R$ and the possibly nonsmooth function $\varphi:\R^p\rightarrow (-\infty, +\infty]$. The function $\varphi$ can capture regularization terms, constraints, and penalties. For instance, when $\varphi$ is an indicator function of a feasible set,
Problem \eqref{eq:P} becomes the distributed constrained optimization problem \cite{beck2017first}. 
The rich choices of $\varphi$ enable broad applications of Problem \eqref{eq:P}, including machine learning \cite{bottou2018optimization, lecun2015deep}, sparse regression \cite{ji2023distributed}, and { control \cite{tron2016distributed,li2015vision}}. 

In this work, we assume $\varphi$ is a weakly convex, lower semicontinuous, proper function, and that each $f_i$ is continuously differentiable on an open set containing $\dom{\varphi}:= \crki{x\in\R^p: \varphi(x)<\infty}$, thereby encompassing most of the aforementioned applications.
We also assume that the proximal operator
\[\proxp{x}:= \arg\min_{y\in\R^{p}} \varphi(y) + \frac{1}{2\gamma}\normi{y-x}^2\in \dom{\varphi},\]
is easy to compute for $\gamma>0$. Typical examples of $\varphi$ include $\ell_1$-norm, discrete information divergence \cite{el2017proximity}, elastic net \cite{zou2005regularization}, among others; see, e.g., \cite{chierchia2020proximity}. 

To solve Problem \eqref{eq:P}, we assume each agent is able to query noisy or stochastic gradients $g_i(x;\xi_i)$ of $\nabla f_i(x)$. This is particularly relevant in data-intensive settings, such as deep learning, where querying full gradients can be computationally prohibitive. 
The stochastic gradients are assumed to satisfy the so-called ABC condition \cite{lei2019stochastic,khaled2022better,huang2023distributed}: 
\begin{equation}
    \label{eq:ABC_full}
    \E_{\xi_{i}}\brk{\norm{g_i(x;\xi_i) - \nabla f_i(x)}^2\middle| x}\leq \C\brk{f_i(x) - f_i^*}  + \sigma^2,
\end{equation}
for some positive constants $\C$ and $\sigma$, where $\xi_i$ denotes some random variable and $f_i^*:= \inf_{x\in\dom{\varphi}} f_i(x)$.
Condition \eqref{eq:ABC_full} generalizes the typical bounded variance condition ($\C = 0$), which is often difficult to verify and can be violated in practice.  
By contrast, the ABC condition applies to more practical scenarios such as sampling a mini-batch of data points to obtain the stochastic gradients \cite{lei2019stochastic,khaled2022better,huang2023distributed}. 
Concrete examples can be found in \cite{khaled2022better,huang2023distributed}.

Despite the importance of solving Problem \eqref{eq:P} with stochastic gradients, the development of relevant algorithms is less advanced, particularly compared to the rich literature that considers $\varphi(x) \equiv 0$.
Furthermore, most existing research suggests that distributed stochastic proximal gradient methods that solve Problem \eqref{eq:P} suffer from slower convergence compared to centralized stochastic proximal gradient method (Prox-SGD) \cite{davis2019stochastic,li2022unified}, especially when the number of agents increases.
For instance, the study in \cite{idrees2024analysis} demonstrates a convergence rate of $\orderi{1/[(1-\lambda)^2 K^{2/3}]}$ for solving Problem \eqref{eq:P} under the Mean Squared Smoothness (MSS) condition \cite{arjevani2023lower} with additional variance reduction procedures, where $K$ stands for the number of iterations, and $1-\lambda$ denotes the spectral gap of the mixing matrix related to the underlying network topology. Another work \cite{li2021decentralized} reports a convergence rate of $\orderi{1/[(1-\lambda)^2 K]}$ for solving Problem \eqref{eq:P} when $f$ is smooth and strongly convex. 
For sparse networks such as ring graphs, $\orderi{1/(1-\lambda)}$ behaves as $\orderi{n^2}$ \cite{nedic2018network}. Therefore, the above two results indicate a significant slowdown in convergence rates for distributed stochastic proximal gradient methods.

Only the recent work \cite{xiao2023one} demonstrates that distributed stochastic proximal gradient methods can benefit from an increasing number of agents and enjoy the so-called ``asymptotic network independence (ANI)'' property \cite{pu2020asymptotic}, that is, the convergence rate of distributed methods matches that of centralized implementation (e.g., $\orderi{1/\sqrt{nK}}$ for smooth 
nonconvex $f$) after a sufficient number of iterations called the \textit{transient time}.
However, the transient time reported in \cite{xiao2023one} behaves as $\orderi{n^3/(1-\lambda)^4}$ when $f$ is smooth and possibly nonconvex, which is much larger than the $\orderi{n^3/(1-\lambda)^2}$ transient time established in previous works that consider $\varphi(x) \equiv 0$ \cite{huang2023cedas,alghunaim2021unified}. Such a discrepancy is particularly pronounced in sparse networks, limiting its applicability in large-scale networks.

We notice that the unsatisfactory transient time reported in \cite{xiao2023one} can be partly attributed to the update scheme of the typical Prox-SGD method: {$x_{k+1}=\prox{\alpha\varphi}\prti{x_k - \alpha g(x_k;\xi_k)}$} \cite{davis2019stochastic}. 
More specifically, Prox-SGD may introduce bias even when the stochastic gradients are unbiased, since we generally have $\E_{\xi_k}\brki{\prox{\alpha\varphi}\prti{x_k - \alpha g(x_k;\xi_k)}| x_k}\neq \prox{\alpha\varphi}\prti{x_k - \alpha \nabla f(x_k)}$ given that $\E_{\xi_k}[g(x_k;\xi_k)| x_k] = \nabla f(x_k)$
at the $k$-th iteration. Such a bias can adversely affect the consensus update in distributed optimization methods, resulting in unsatisfactory performance.
To address this issue, we apply the \textit{normal map} update scheme \cite{robinson1992normal,milzarek2023convergence,ouyang2024trust} to distributed stochastic proximal gradient methods. The normal map update with stochastic gradient $g(x_k;\xi_k)$ performs the following core step:
\begin{equation}
    \label{eq:intro_np}
    \begin{aligned}
        z_{k + 1} &= z_k - \alpha \brk{g(x_k;\xi_k) + \gamma^{-1}\prt{z_k - x_k}},\\
        x_{k + 1} &= \proxp{z_{k + 1}}.
    \end{aligned}
\end{equation}
It {can be seen} that \eqref{eq:intro_np} preserves the unbiasedness since
\begin{align*}
    &\E_{\xi_k} \crk{z_k - \alpha \brk{g(x_k;\xi_k) + \gamma^{-1}\prt{z_k - x_k}}\middle| x_k,z_k}\\
    & = z_k - \alpha \brk{\nabla f(x_k) + \gamma^{-1}\prt{z_k - x_k}},
\end{align*}
given that $\E_{\xi_k} [g(x_k;\xi_k)\mid x_k] = \nabla f(x_k)$.

Inspired by the observation above, we introduce a unified algorithmic framework for distributed stochastic proximal gradient methods using the normal map update, termed Normal Map-based Stochastic ABC-2 (norM-SABC-2), and propose two algorithms: Normal Map-based Distributed Stochastic Gradient Tracking (\normgt) and Normal Map-based Exact Diffusion (\normed), within this framework. The proposed methods 
avoid the bias introduced by distributed Prox-SGD method and { reduce} the transient times for solving Problem \eqref{eq:P} compared to the earlier work \cite{xiao2023one} { without additional communication or notable computational overhead}. In particular, the \normed~method attains the same transient time as its non-proximal counterpart in \cite{huang2023cedas,alghunaim2021unified}, { which is the shortest known for solving Problem~\eqref{eq:P}.}

\subsection{Related Works}

We focus on the recent advances in distributed composite optimization that considers Problem \eqref{eq:P}. 
The study of Problem~\eqref{eq:P} with full gradients available dates back to the seminal work \cite{nedic2010constrained}, where $\varphi$ represents the indicator function of a common constrained set. Subsequent works including \cite{zeng2018nonconvex,chen2012fast} expand to the case where $\varphi$ is a convex function. However, the proposed algorithms are based on the Distributed Gradient Descent (DGD) method \cite{nedic2009distributed} which does not converge to the exact solution under a constant stepsize. 
To improve the algorithmic convergence, there are mainly two lines of works, including \cite{di2016next,ye2020decentralized} that apply the gradient tracking technique \cite{xu2015augmented,di2016next,nedic2017achieving}, and primal-dual type methods including \cite{shi2015proximal,li2019decentralized,alghunaim2019linearly,guo2023decentralized} which are inspired by earlier works assuming $\varphi=0$ \cite{shi2015extra,yuan2018exact}.
Specifically, the method in \cite{ye2020decentralized} utilizes multiple inner loops of communication at every iteration
to obtain the optimal complexity when the smooth component $f$ is strongly convex. 
The work \cite{di2016next} applies the proximal step to a surrogate function and proves the asymptotic convergence for the considered method. Regarding the primal-dual type algorithms, those in \cite{shi2015proximal,li2019decentralized,guo2023decentralized} allow for different nonsmooth components across the agents and enjoy the sublinear convergence when $f$ is convex. 
The paper \cite{alghunaim2019linearly} establishes the linear convergence for the proposed method under a shared $\varphi$ and strongly convex smooth component $f$. In \cite{xu2021distributed,alghunaim2020decentralized}, both gradient tracking-based and primal-dual type methods are unified under a shared $\varphi$, and linear convergence rates are demonstrated for strongly convex $f$. 

When only noisy gradients of $f_i$ are available, the study on distributed stochastic proximal gradient methods is less advanced. The early work in \cite{bianchi2012convergence} establishes the asymptotic convergence of the distributed projected stochastic gradient method. When large mini-batches are available, the authors in \cite{wang2021distributed} prove convergence with momentum updates. The work in \cite{xin2021stochastic} demonstrates a topology-independent sample complexity, provided that large mini-batches of data samples and multiple inner loops of communication are utilized. However, the requirement for large mini-batches or multiple inner loops of communication can be impractical in real-world applications.
The works in \cite{yan2023compressed,li2021decentralized} employ communication compression { and achieve sublinear convergence rates, but do not demonstrate the ANI property.}
To our knowledge, only \cite{xiao2023one,olshevsky2022asymptotic} demonstrate the ANI property. Notably, the methods proposed in \cite{xiao2023one} involve momentum-like update and require communicating additional variables. The work \cite{olshevsky2022asymptotic} establishes the ANI property for distributed stochastic subgradient methods under specific stepsizes while the transient time remains unknown.

\subsection{Main Contribution}

The main contribution of this paper is three-fold.

Firstly, { we introduce norM-SABC-2, the first unified algorithmic framework for normal map-based distributed stochastic proximal gradient methods, and propose two algorithms within it: \normgt~and \normed.} 
Both methods exhibit the asymptotic network independence (ANI) property with improved transient times. In particular, under smooth and possibly nonconvex $f$ and weakly convex $\varphi$, \normgt~reduces the state-of-the-art transient time from $\orderi{n^3/(1-\lambda)^4}$ to {$\orderi{\max\crki{n^3/(1-\lambda)^2,n/(1-\lambda)^3}}$}, and \normed~reduces the transient time to $\orderi{n^3/(1-\lambda)^2}$. Notably, \normed~is the first distributed stochastic proximal gradient method to match the transient time of the non-proximal ED algorithm. { In addition, \normgt~matches the transient time of DSGT and norM-ED under the condition $(1-\lambda)^{-1}=\orderi{n^{2}}$, which is generally satisfied in practice \cite{nedic2018network}}. { See Table~\ref{tab:comp} for a detailed comparison.}

Secondly, we make less restrictive assumptions compared to the previous works to obtain the enhanced results. Specifically, this paper assumes only weak convexity on $\varphi$ and the ABC condition on the stochastic gradients, and does not impose any additional data heterogeneity conditions among the individual functions $f_i$.
Given the smoothness of $f$, these relaxed assumptions represent the most general conditions for solving Problem~\eqref{eq:P}, even considering the centralized setting.

Thirdly, we explore a flexible analytical framework for distributed stochastic proximal gradient methods, which accommodates multi-step analysis with an arbitrary number of steps $m\geq 1$. Compared to previous works that are tailored to specific algorithms, our analysis encloses a wide range of distributed stochastic proximal gradient algorithms, provided that their averaged iterates update as a centralized method with some additional error terms. Moreover, by choosing $m$ appropriately, the framework extends naturally to other problem settings, including utilizing the random reshuffling strategy for sampling stochastic gradients and optimization over time-varying networks.

\begin{table}[htbp]
    \centering
    \setlength{\tabcolsep}{1pt}
    {
\begin{tabular}{@{}ccccc@{}}
\toprule
Method                                    & $\varphi$     & Var          & Transient Time                                                                     & \makecell[c]{Transient Time \\ (Ring graph)} \\ \midrule
DEPOSITUM \cite{zhou2025decentralized}    & WX          & BD           & /                                                                                  & /                                            \\
DProxSGT \cite{yan2023compressed}         & CX           & BD           & /                                                                                  & /                                            \\
Prox-DASA(-GT) \cite{xiao2023one}         & CX           & BD           & $\order{\frac{n^3}{(1-\lambda)^4}}$                                                & $\order{n^{11}}$                             \\ \midrule 
DSGT \cite{alghunaim2021unified} & / & BD & $\order{\frac{n^3}{(1-\lambda)^2} + \frac{n}{(1-\lambda)^{8/3}}}$ & $\order{n^{7}}$ \\ 
\makecell{\textbf{\normgt}\\ (This work)} & \textbf{WX} & \textbf{ABC} & $\order{\bs{\frac{n^3}{(1-\lambda)^2}}+\bs{\frac{n}{(1-\lambda)^{3}}}}$ & $\order{\bs{n^{7}}}$                         \\ \midrule
ED \cite{alghunaim2021unified} & / & BD & $\order{\frac{n^3}{(1-\lambda)^2}}$ & $\order{n^{7}}$ \\
\makecell{\textbf{\normed}\\ (This work)} & \textbf{WX} & \textbf{ABC} & $\order{\bs{\frac{n^3}{(1-\lambda)^2}}}$                                           & $\order{\bs{n^{7}}}$                         \\ 
\bottomrule
\end{tabular}
\caption{ Comparison of assumptions and transient time performance. The ``$\varphi$'' column indicates assumptions on the nonsmooth function $\varphi$ ( ``CX'': convex; ``WX'': weakly convex). The ``Var'' column specifies the variance condition on stochastic gradients (``BD": bounded variance condition; ``ABC": ABC condition). Ring graph results correspond to $1-\lambda=\orderi{n^{-2}}$. 
    }
    \label{tab:comp}}
\end{table}

\subsection{Notations and Assumptions}

Throughout this paper, column vectors are considered by default unless specified otherwise. Let $x_{i,k},y_{i,k},z_{i,k} \in \R^p$ represent the iterates of agent $i\in\cN:=\{1,2,\ldots,n\}$ at the $k$-th iteration. { We introduce the following stacked variables; the others are defined analogously.}
\begin{align*}
    \x_k&:= (x_{1,k}, x_{2, k},\ldots, x_{n,k})^{\T}\in\R^{n\times p},\\
    \nabla F (\x_k) &:= \prt{\nabla f_1(x_{1,k}),\ldots, \nabla f_n(x_{n,k})}^{\T}\in\R^{n\times p},\\
    \bproxp{\z_k}&:= (
        \proxp{z_{1, k}},
        \ldots, \proxp{z_{n, k}}
    )^{\T}\in\R^{n\times p}.
\end{align*}
The average of $x_i$ across all the agents is denoted as $\bar{x} \in \R^p$.
We use $\normi{\cdot}$ to denote the Frobenius norm for a matrix and the $\ell_2$-norm for a vector by default. The inner product of two vectors $a, b\in\R^{p}$ is written as $\inproi{a, b}$. For two matrices $A, B\in\R^{n\times p}$, the inner product $\inproi{A, B}$ is defined as $\inproi{A, B} := \sum_{i=1}^n\inproi{A_i, B_i}$, where $A_i$ (and $B_i$) represents the $i$-row of $A$ (and $B$).

We now introduce the standing assumptions considered in this paper. 

Assumption \ref{as:smooth} is standard in the literature \cite{huang2024accelerated,huang2023distributed} when $\varphi(x)\equiv 0$. 
\begin{assumption}
    \label{as:smooth}
    Each $f_i:\R^p\rightarrow\R$ is $L$-smooth on an open set $\U$ containing $\dom\varphi$, meaning that
    \begin{align*}
        \norm{\nabla f_i(x) - \nabla f_i(x')}\leq L\norm{x - x'}, \ \forall x,x'\in \U.
    \end{align*}
    In addition, each $f_i$ is bounded below, i.e., $f_i(x)\geq \uf_i := \inf_{x\in\dom{\varphi}} f_i(x)>-\infty, \forall x\in\R^p$. We also denote $f^*:= \inf_{x\in\dom{\varphi}} f(x)$.
\end{assumption}

Assumption \ref{as:phi} considers a more general condition compared to previous works (see e.g., \cite{xiao2023one,shi2015proximal}), where $\varphi$ is $\rho$-weakly convex, i.e., $\varphi(x) + \rho\normi{x}^2/2$ is convex for some $\rho > 0$. 

\begin{assumption}
    \label{as:phi}
    The function $\varphi: \R^{p} \rightarrow (-\infty, \infty]$ is $\rho$-weakly convex, lower semicontinuous, proper, { and bounded below}. We denote $\varphi^*:= \inf_{x\in\dom{\varphi}} \varphi(x)$.
\end{assumption}

We consider the ABC condition on the stochastic gradients in Assumption \ref{as:abc}, which is the most general variance condition under the distributed optimization setting \cite{huang2024accelerated,huang2023distributed,khaled2022better}. 
\begin{assumption}
    \label{as:abc}
    Each agent has access to a conditionally unbiased stochastic gradient $g_i(x;\xi_i)$ of $\nabla f_i(x)$, i.e., $\E_{\xi_i}\brki{g_i(x;\xi_{i})| x} = \nabla f_i(x)$, and there exist constants $\C\geq0$ and $\sigma\geq 0$ such that for any $k\in\mathbb{N}$, $i\in\cN$, and $\forall x\in\U$, 
    \begin{align}
        \label{eq:abc}
        \E_{\xi_i}\brk{\norm{g_{i}(x;\xi_{i}) - \nabla f_i(x)}^2\middle|{x}}&\leq \C\brk{f_i(x) - \uf_i} + \sigma^2
    \end{align} 
    for an open set $\U$ containing $\dom\varphi$. In addition, the stochastic gradients are independent across different agents at each $k\geq 0$. 
\end{assumption}

{ Under Assumptions \ref{as:smooth}--\ref{as:abc}, Problem~\eqref{eq:P} covers a broader class of problems than existing works such as \cite{xiao2023one,alghunaim2021unified,milzarek2023convergence,ouyang2024trust}. Specifically, $\varphi$ can be chosen as the smoothly clipped absolute deviation (SCAD) regularizer \cite{fan2001variable}, which is weakly convex but not convex, and $f_i$ can be set as in \cite[Section VI-A]{huang2023distributed} to satisfy the ABC condition while violating the bounded variance assumption.
}

{ The agents communicate over} {an undirected} graph $\mathcal{G} = (\cN, \mathcal{E})$ with $\mathcal{E}\subseteq \cN\times \cN$ representing the set of edges. In particular, $(i,i)\in\mathcal{E}$ for all $i\in\mathcal{N}$. The set of neighbors for agent $i$ is denoted by $\mathcal{N}_i=\{j\in \mathcal{N}:(i,j)\in \mathcal{E}\}$.
The element $w_{ij}$ in the weight matrix $W\in\mathbb{R}^{n\times n}$ represents the weight of the edge between agents $i$ and $j$. 
{ Consequently}, we consider Assumption \ref{as:graph} that is standard in the distributed optimization literature. The condition guarantees that the spectral norm $\lambda$ of the matrix $(W - \1\1^{\T}/n)$ is strictly less than one. 

\begin{assumption}
    \label{as:graph}
    The graph $\mathcal{G}$ is undirected and {connected}, i.e., there exists a path between any two {distinct} nodes in $\mathcal{G}$. There is a direct link between $i$ and $j$ $(i\neq j)$ in $\mathcal{G}$ if and only if $w_{ij}>0$ and $w_{ji}>0$; otherwise, $w_{ij}=w_{ji}=0$. The mixing matrix $W$ nonnegative, stochastic, and symmetric, i.e., $\1^{\T}W=\1^{\T}$ and $W^{\T}=W$. 
\end{assumption}



\section{Normal Map-based Distributed Stochastic Proximal Gradient Methods}
\label{sec:methods}

In this section, we first introduce the definition of \textit{normal map} $\Fnor{\cdot}:\R^p\rightarrow \R^p$, as established in \cite{robinson1992normal}, along with a critical property~\eqref{eq:Fnor_unbiased} in subsection~\ref{subsec:motivation}. Such a property further inspires the unified framework stated in \eqref{eq:norm-abc}. Subsection~\ref{subsec:algs} then presents the proposed algorithms within the unified framework.  

\subsection{Motivation and Unified Framework}
\label{subsec:motivation}

The normal map $\Fnor{\cdot}:\R^p\rightarrow\R^p$ is defined as:
\begin{equation}
    \label{eq:Fnor}
    \begin{aligned}
        \Fnor{z}&:= \nabla f(\proxp{z}) + \frac{1}{\gamma}\prt{z - \proxp{z}}\\
        &\in \partial \psi\prt{\proxp{z}},\; \forall z\in\R^p,
    \end{aligned}
\end{equation}
{where
$\partial\psi(x)$ denotes the subdifferential set of $\psi(x)$ satisfying}
$\partial\psi(x) = \nabla f(x) + \partial\varphi(x)$ \cite{attouch2013convergence}. The normal map $\Fnor{\cdot}$ {provides} a subgradient of the objective function $\psi(\cdot)$. Notably, given an unbiased stochastic gradient $g(x;\xi)$ of $\nabla f(x)$, the normal map $\Fnor{\cdot}$ preserves the unbiasedness:
\begin{equation}
    \label{eq:Fnor_unbiased}
        \E_{\xi}\brk{g(\proxp{z};\xi) + \gamma^{-1}\prt{z - \proxp{z}}{\mid z}} = \Fnor{z},
\end{equation}
for any $z\in\R^p$.
Such a property enables the use of existing analytical tools developed for the full gradient case. By contrast, the unbiasedness of $g(x;\xi)$ is not preserved for traditional stochastic proximal gradient methods, since $\E_\xi \brki{\proxp{x - \gamma g(x;\xi)}{\mid x}} \neq \proxp{x - \gamma \nabla f(x)}$ in general. As a result, the use of previous analytical tools for proximal gradient methods may introduce additional errors into the variance term $\E_\xi\brki{\normi{g(x;\xi) - \nabla f(x)}^2{\mid x}}$ and further results in unsatisfactory convergence results for existing distributed stochastic proximal gradient methods. By considering \eqref{eq:Fnor_unbiased} in developing and analyzing new algorithms, the additional term $1/(1-\lambda)$ in the variance term can be avoided (see the proof of Lemma \ref{lem:norm-abc_cons_tw}), thereby improving the transient times.

From the above discussion, it is natural to define the normal map for each local function $\psi_i:= f_i + \varphi$ as{$\Fnori{i}{z}:= \nabla f_i(\proxpi{z}) + \gamma^{-1}\prti{z - \proxpi{z}}$.}
Then, the update of each agent applies such a subgradient by replacing $\nabla f_i(\proxpi{z})$ with the stochastic gradient $g_i(\proxpi{z};\xi)$. This {provides the} basis for { the following} unified framework: 
{
\begin{equation}
    \label{eq:norm-abc}
    \begin{aligned}
        \x_k& =\bproxp{\z_k},\\
        \z_{k + 1} &= A\crk{C\z_k - \alpha\brk{\g_k + \gamma^{-1}\prt{\z_k - \x_k}}} - B\d_k\\
        & = A\crk{C \z_k - \alpha\brk{\Fnorb{\z_k} + \g_k - \nabla F(\x_k)}} -B\d_k,\\
        \d_{k + 1} &= \d_k + B\z_{k + 1},\; k=0,1,\ldots
    \end{aligned}
\end{equation}
}
{where the initialization of $\d_0$ depends on the choice of the matrices $A$, $B$, and $C$, as well as the initial point $\z_0$.} 
Here, $\Fnorb{\z_k} = \nabla F(\x_k) + \gamma^{-1}\prti{\z_k - \x_k}$. When $\varphi=0$, thus $\Fnorb{\z_k}=\nabla F(\z_k)$, the method in~\eqref{eq:norm-abc} reduces to the one considered in \cite{alghunaim2021unified}.
The unified framework \eqref{eq:norm-abc} is named as Normal Map-based Stochastic ABC-2 (norM-SABC-2), where ABC-2 refers to the existing unified algorithms employing matrices $A$, $B$ and $C$ \cite{xu2021distributed,alghunaim2020decentralized,alghunaim2021unified} and the ABC condition concerning the stochastic gradients \eqref{eq:abc}.
The matrices $A$, $B$, and $C$ are assumed to satisfy {the following assumption}.

\begin{assumption}
    \label{as:abcW}
    The matrices $A, B^2,C\in\R^{n\times n}$ are chosen as a polynomial function of $W$: $ A=\sum_{d=0}^{s}a_d W^d$, $B^2=\sum_{d=0}^{s}b_d W^d$, $C=\sum_{d=0}^{s}c_d W^d$, 
        where $s\geq 0$ is an integer. Moreover, the constants $\{a_d,b_d,c_d\}_{d=0}^s$ are chosen such that $A$ and $C$ are doubly stochastic and the matrix $B$ satisfies $B\x=\0$ if and only if $x_1=x_2=\ldots=x_n$.
\end{assumption}

Choosing different matrices $A$, $B$, and $C$ {in Assumption~\ref{as:abcW}} leads to specific normal map-based distributed stochastic proximal gradient methods.
{Assumption~\ref{as:abcW}} is satisfied by most distributed gradient-based methods.

\subsection{Proposed Algorithms}
\label{subsec:algs}

We introduce two novel algorithms within the proposed framework \eqref{eq:norm-abc}: Normal Map-based Distributed Stochastic Gradient Tracking (\normgt) and Normal Map-based Exact Diffusion (\normed). Specifically, { the \normgt~method outlined in Algorithm \ref{alg:norm-dsgt} is an instance of \eqref{eq:norm-abc} with $A = W$, $B = I-W$, $C=W$, and $\d_0 = -W\z_0$. {This can be seen by comparing the expressions for $\z_{k+1} - \z_k$ in Algorithm \ref{alg:norm-dsgt} and \eqref{eq:norm-abc}, after eliminating $\y_k$ and $\d_k$, respectively. The verification for norM-ED given later is analogous.}} At each iteration, agent $i$ first performs an approximate subgradient step and then communicates with its neighbors to obtain the variable $z_{i, k + 1}$ in Line \ref{line:norm-dsgt_wz}. A proximal step is then applied to $z_{i, k + 1}$ to compute the next iterate, $x_{i, k + 1}$ (Line \ref{line:norm-dsgt_x}). The stochastic gradient is evaluated at $x_{i,k + 1}$ in Line \ref{line:norm-dsgt_g} to maintain the relation in \eqref{eq:Fnor_unbiased}. { The initialization in Line~\ref{line:norm-dsgt_y0} and the update in} Line \ref{line:norm-dsgt_y} are applied to track the overall stochastic subgradients (normal maps){, i.e., $\bar{y}_k = \bar{g}_k + \gamma^{-1}\prt{\bar{z}_k - \bar{x}_k}$, where $\bar{g}_k:= \sumn g_{i,k}/n$, $\bar{z}_k := \sumn z_{i,k}/n$, and $\bar{x}_k := \sumn x_{i,k}/n$}. { Notably, \normgt\ shares the same communication overhead as the non-proximal DSGT algorithm, and saves one round of communication per iteration compared to Prox-DASA-GT \cite{xiao2023one}.}

\begin{algorithm}
	\begin{algorithmic}[1]
		\STATE Initialize $z_{i,0}\in\R^p$ for all agent $i\in\mathcal{N}$, determine $W = [w_{ij}]\in\R^{n\times n}$, stepsize {$\alpha>0$ and parameter $\gamma>0$}. 
        \STATE Calculate $x_{i,0} = \proxp{z_{i,0}}$.
		\FOR{$k=0, 1, 2, ..., K-1$}
		\FOR{Agent $i = 1, 2, ..., n$ in parallel}
        \IF{$k=0$}
        \STATE Query $g_{{i}, 0} = g_i(x_{i, 0};\xi_{i, 0})\in\R^p$.\label{line:norm-dsgt_g0}
        \STATE Initialize $y_{i,0} = g_{i,0} + \gamma^{-1}\prti{z_{i,0} - x_{i,0}}$. \label{line:norm-dsgt_y0}
        \ENDIF
		\STATE Update $z_{i, k + 1} = \sum_{j\in\cN_i}w_{ij} \prti{z_{j,k} - \alpha y_{j,k}}$. \label{line:norm-dsgt_wz}
        \STATE Update $x_{i,k + 1} = \proxp{z_{i,k + 1}}$. \label{line:norm-dsgt_x}
        \STATE Query $g_{{i}, k + 1} = g_i(x_{i, k + 1};\xi_{i, k + 1})\in\R^p$.\label{line:norm-dsgt_g}
        \STATE Update $y_{i, k + 1} = \sum_{j\in\cN_i}w_{ij} y_{j,k} + \brki{g_{i, k + 1} + \gamma^{-1}\prti{z_{i,k + 1} - x_{i,k + 1}}}  - \brki{g_{i,k} + \gamma^{-1}\prti{z_{i,k} - x_{i,k}}}$. \label{line:norm-dsgt_y}
		\ENDFOR
		\ENDFOR
	\end{algorithmic}
	\caption{Normal Map-based Distributed Stochastic Gradient Tracking (\normgt)}
	\label{alg:norm-dsgt}
\end{algorithm}

{ The \normed~method outlined in Algorithm \ref{alg:norm-ed} is an instance of \eqref{eq:norm-abc} with $A = W$, $B = (I-W)^{1/2}$, $C = I$, and $\d_0 = \0$.}  
In \normed, agent $i$ performs a local update in Line \ref{line:norm-ed_z}, which can be regarded as a combination of a local subgradient step $a_{i,k+1} = z_{i,k} - \alpha\brki{g_{i,k} + \gamma^{-1}\prti{z_{i,k} - x_{i,k}}}$ and a correction step $z_{i, k + \frac{1}{2}} = z_{i,k} + a_{i,k + 1} - a_{i, k}$ according to \cite{yuan2018exact}. Notably, { \normed\ shares the same communication overhead as its non-proximal ED algorithm, and has lower communication overhead compared to \normgt.} 

\begin{algorithm}
	\begin{algorithmic}[1]
		\STATE Initialize $z_{i,0}\in\R^p$ for all agent $i\in\mathcal{N}$, determine $W = [w_{ij}]\in\R^{n\times n}$, stepsize {$\alpha>0$ and parameter $\gamma>0$}. 
        \STATE Calculate $x_{i,0} = \proxp{z_{i,0}}$.
		\FOR{$k=0, 1, 2, ..., K-1$}
		\FOR{Agent $i = 1, 2, ..., n$ in parallel}
        \STATE Query $g_{{i}, k} = g_i(x_{i, k};\xi_{i, k})\in\R^p$.\label{line:norm-ed_g0}
        \IF{$k=0$}
        \STATE Update $z_{i,\frac{1}{2}} = z_{i,0} - \alpha \brk{g_{i,0} + \gamma^{-1}\prti{z_{i,0} - x_{i,0}}}$.\label{line:norm-ed_k0}
		\ELSE 
        \STATE Update $z_{i, k + \frac{1}{2}} = 2z_{i,k} - z_{i,k-1} - \alpha\brki{g_{i,k} + \gamma^{-1}\prti{z_{i,k} - x_{i,k}}} + \alpha\brki{g_{i,k-1} + \gamma^{-1}\prti{z_{i,k-1} - x_{i,k-1}}}$. \label{line:norm-ed_z}
        \ENDIF
		\STATE Update $z_{i, k + 1} = \sum_{j\in\cN_i}w_{ij} z_{j, k + \frac{1}{2}}$. \label{line:norm-ed_wz}
        \STATE Update $x_{i,k + 1} = \proxp{z_{i,k + 1}}$. \label{line:norm-ed_x}
		\ENDFOR
		\ENDFOR
	\end{algorithmic}
	\caption{Normal Map-based Exact Diffusion (\normed)}
	\label{alg:norm-ed}
\end{algorithm}

The rationale behind the effectiveness of \normgt~and \normed~can be seen from the update \eqref{eq:zbark} for the averaged iterates $\bar{x}_k$ and $\bar{z}_k$ (see Lemma \ref{lem:re}):
\begin{eqnarray}\label{eq:zbark}
        \bar{z}_{k + 1} 
        &\!\! =\!\! & \bar{z}_k - \alpha \Fnor{\bar{z}_k} - \frac{\alpha}{n}\sumn\brk{g_i(x_{i,k};\xi_{i,k}) - \nabla f_i(x_{i,k})}\cr
        && -\frac{\alpha}{n}\sumn\brk{\Fnori{i}{z_{i,k}}- {\Fnori{i}{\bar{z}_k}}},\\
    \bar{x}_{k + 1} &\!\!=\!\! &\frac{1}{n}\sumn\proxp{z_{i, k + 1}}\notag.
\end{eqnarray}
From \eqref{eq:zbark}, both proposed methods can be viewed as approximate implementations of the centralized norM-SGD method described in \cite{milzarek2023convergence}:
\begin{equation*}
    \begin{aligned}
        z_{k + 1} &= z_k - \alpha \Fnor{z_k} + \alpha \brk{\nabla f(x_k) - g(x_k;\xi_k)},\\
        x_{k + 1} &= \proxp{z_{k + 1}}.
    \end{aligned}
\end{equation*}

The update for $\bar{z}_k$ in \eqref{eq:zbark} resembles that of the norM-SGD method. However, the update for $\bar{x}_k$ differs due to the nonlinearity of the proximal operator, in other words, 
\begin{equation}
    \label{eq:proxp_non}
    \frac{1}{n}\sumn\proxp{z_{i,k}}\neq \proxp{\frac{1}{n}\sumn z_{i,k}},\; \forall k\geq 0.
\end{equation}
The introduced analytical challenges will be addressed by Lemmas \ref{lem:fbarx2fproxz} and \ref{lem:pix2piz} in the next section. Intuitively, the discrepancy described in \eqref{eq:proxp_non} can be bounded by the consensus error term, $\sumn\normi{z_{i,k} - \bar{z}_k}^2/n$ thanks to Assumptions~\ref{as:phi}~and~\ref{as:graph}. { This also highlights the differences between the present work and both the centralized normal map-based analyses in \cite{milzarek2023convergence,ouyang2024trust} and the existing distributed analyses in \cite{alghunaim2021unified}. 
Relative to the centralized analyses \cite{milzarek2023convergence,ouyang2024trust}, the distributed setting introduces two additional technical challenges.  
First, due to the noncommutativity in \eqref{eq:proxp_non}, the averaged iterate $\bar{x}_k = \sumn\proxp{z_{i,k}}/n$
cannot be directly compared with 
the centralized iterate $\proxp{\bar{z}_k}$. This discrepancy gives rise to additional error terms that depend on the consensus error 
\[\frac{1}{n}\sumn\norm{z_{i,k} - \bar{z}_k}^2.\]
Second, controlling this consensus error is challenging because it is coupled with errors from three sources: (i) stochastic proximal gradient descent, (ii) stochastic gradients under the ABC condition, and (iii) the multi-step analysis itself.

Compared with existing distributed analyses \cite{alghunaim2021unified,xiao2023one}, the present work must handle the proximal term and the ABC condition simultaneously within a multi-step analysis framework. In particular, unlike the bounded variance condition ($\C = 0$), the ABC condition creates a circular dependence between the function value gap and the consensus error term. This coupled treatment differs fundamentally from previous one-step distributed analyses \cite{alghunaim2021unified,xiao2023one}.}

\section{Convergence Analysis}
\label{sec:ana}

We present the convergence analysis for the proposed unified algorithmic framework in this section. Specifically, we state the preliminary results along with the stationarity measure \eqref{eq:prox-grad} in Subsection \ref{subsec:pre} followed by the main analysis in Subsection \ref{subsec:multi}. 

\subsection{Preliminaries}
\label{subsec:pre}

In this part, we introduce several preliminary results that facilitate the convergence analysis in the subsequent subsection. Generally speaking, these preliminaries help establish the relationship between the two sequences of variables $\crki{z_{i,k}}_{i=1}^n$ and $\crki{x_{i,k} = \proxpi{z_{i,k}}}_{i=1}^n$ for $k\geq 0$, { allowing the subsequent analysis to focus on} $\crki{z_{i,k}}_{i=1}^n$ only.

In Lemma~\ref{lem:re}, we present the transformed form of \eqref{eq:norm-abc} { through the eigenvalue decomposition of $W = \hU \hLambda \hU^{\T} + \1\1^{\T}/n$,} 
where $1=\lambda_1>\lambda_2\ge\lambda_3\ge\cdots\ge\lambda_n$ are the eigenvalues of $W$, $\hLambda:= \diag(\lambda_2,\ldots, \lambda_n)$, $\hU\hU^{\T}=I-\1\1^{\T}/n$, and $\hU^{\T}\hU=I_{n-1}$.  
Here, $I\in\R^{n\times n}$ and $I_{n-1}\in\R^{(n-1)\times (n-1)}$ are both identity matrices, and $\1\in\R^{n}$ is the vector with all one elements.

\begin{lemma}
    \label{lem:re}
    Let Assumptions \ref{as:smooth}, \ref{as:phi}, \ref{as:graph}, and \ref{as:abcW} hold. { Denote $\delta_{i,k}:= g_{i}(x_{i,k};\xi_{i,k}) - \nabla f_i(x_{i,k})$ and $\Delta_{k}:= \g_k - \nabla F(\x_k)$.}
    We have
    \begin{subequations}
    \label{eq:transformed}
        \begin{align}
         \bar{z}_{k + 1} &= \bar{z}_k- \alpha \Fnor{\bar{z}_k} - \frac{\alpha}{n}\sumn \brk{\Fnori{i}{z_{i,k}} - {\Fnori{i}{\bar{z}_k}}}  \nonumber\\
         &\quad - \frac{\alpha}{n}\sumn{ \delta_{i,k}},\text{ and } \label{eq:abc_zbark}\\
        \bu_{k + 1} &= \Gamma \bu_k  -\alpha { [V^{-1}]_l
            \hLambda_a\hU^{\T}\brk{\Fnorb{\z_k} - \Fnorb{\barz_k}}}\nonumber\\
        &-\alpha { [V^{-1}]_r 
            \hLambda_b^{-1} \hLambda_a U^{\T}\brk{\Fnorb{\barz_{k + 1}} - \Fnorb{\barz_k}}}\nonumber\\
            & -\alpha { [V^{-1}]_l\hLambda_a\hU^{\T}\Delta_k}, \label{eq:buk}
    \end{align}
    \end{subequations}
    for every iteration $k\geq0$. The matrix $\Gamma\in\R^{2(n-1)\times 2(n-1)}$ is { assumed to admit an eigenvalue decomposition as follows:}
    \begin{align}
        \label{eq:Gamma}
        \Gamma := V^{-1}\begin{pmatrix}
            \hLambda_a\hLambda_c - \hLambda_b^2 & - \hLambda_b \\
            \hLambda_b & I_{n-1}
        \end{pmatrix}V,
    \end{align}
    where ${ V^{-1}:= ([V^{-1}]_l, [V^{-1}]_r)}$ is some invertible matrix with {$[V^{-1}]_l$, $[V^{-1}]_r\in\R^{2(n-1)\times (n-1)}$}, and $\hat{\Lambda}_a$, $\hat{\Lambda}_b^2$ and $\hat{\Lambda}_c$ are diagonal matrices composed of eigenvalues of $A$, $B^2$, and $C$, respectively: $\hat{\Lambda}_a=\diag(\lambda_{a,2}, \ldots, \lambda_{a,n})$, $\hat{\Lambda}_b^2=\diag(\lambda_{b,1}^2, \ldots, \lambda_{b,n-1}^2)$, and $\hat{\Lambda}_c=\diag(\lambda_{c,2}, \ldots, \lambda_{c,n}).$
    In addition, {$\bu_k:= [V^{-1}]_l\hU^{\T}\z_k + [V^{-1}]_r\hLambda_b^{-1}\hU^{\T}\d_k'\in\R^{2(n-1)\times p}$ and $\d_k':= B\prti{\d_k - B\z_k} + \alpha A\Fnorb{\barz_k}\in \R^{n\times p}$}.
    
\end{lemma}

\begin{remark}
    \label{rem:cons_beta}
     The quantity $\normi{\bu_k}^2$ is closely related to the consensus error term $\normi{\z_k  - \1\bar{z}_k^{\T}}^2$ due to $\normi{\z_k  - \1\bar{z}_k^{\T}}^2\leq \normi{V\bu_k}^2$. 
     In addition, we define $\beta\in\R$ as $\beta:= \normi{\Gamma}$
     and assume in the subsequent analysis that $\beta\in(0,1)$. Indeed, such an assumption { and the eigenvalue decomposition in \eqref{eq:Gamma} are} satisfied for many distributed algorithms \cite{alghunaim2021unified}, including \normgt~and~\normed.\footnote{ For both methods, we have $\beta\sim\orderi{\lambda}$.}
\end{remark}

\begin{proof}
    See Appendix \ref{app:re}.
\end{proof}

Lemma \ref{lem:Fnor_Lip} states that both $\Fnor{\cdot}$ and $\Fnori{i}{\cdot}$ are Lipschitz continuous given that $\gamma \in (0, 1/\rho)$ under Assumptions \ref{as:smooth} and~\ref{as:phi}.

\begin{lemma}
    \label{lem:Fnor_Lip}
    Let Assumptions \ref{as:smooth} and \ref{as:phi} hold. Set $\gamma < 1/\rho$. The normal maps $\Fnor{\cdot}:\R^p\rightarrow\R^p$ and $\Fnori{i}{\cdot}: \R^p\rightarrow\R^p$ are $\Ln$-Lipschitz continuous, where $\Ln:= \prti{L - \rho + 2/\gamma}/\prti{1 - \gamma\rho}$.
\end{lemma}

\begin{proof}
    See Appendix \ref{app:Fnor_Lip}.
\end{proof}

Lemma \ref{lem:fbarx2fproxz} relates $[f(\bar{x}_k) - f^*]$ to $[f(\proxp{\bar{z}_k}) - f^*]$, characterizing the impact due to the nonlinearity of the proximal operator. The function value $f(\bar{x}_k)$ arises from the ABC condition \eqref{eq:abc}. Consequently, the subsequent analysis can focus on the changes in $\psi\prti{\proxp{\bar{z}_k}} = f(\proxp{\bar{z}_k}) + \varphi(\proxp{\bar{z}_k})$, as detailed in Lemma \ref{lem:cH_descent} later. 

\begin{lemma}
    \label{lem:fbarx2fproxz}
    Let Assumptions \ref{as:smooth} and \ref{as:phi} hold. We have $
    f(\bar{x}_k) - f^* \leq 2\brki{f(\proxp{\bar{z}_k}) - f^* } 
    + L\normi{\bar{x}_k - \proxp{\bar{z}_k}}^2$ for any $k\geq 0$.
\end{lemma}

\begin{proof}
    Applying the descent lemma and invoking $\normi{\nabla f(x)}^2\leq 2L(f(x) - f^*)$ { (see, e.g., \cite[Theorem 2.1.5]{nesterov2018lectures})} yields the desired result.
\end{proof}

Lemma \ref{lem:pix2piz} bounds
the consensus error $\normi{\Pi\x_k}^2$ ($\Pi:= I - \1\1^{\T}/n$) 
and the error term $\normi{\bar{x}_k - \proxp{\bar{z}_k}}^2$ within $\normi{\Pi\z_k}^2$. Consequently, we can handle all such errors based on the update of $\bu_k$ in Lemma \ref{lem:re}.

\begin{lemma}
    \label{lem:pix2piz}
    Let Assumptions \ref{as:smooth}, \ref{as:phi} and \ref{as:graph} hold. We have for all $k\geq 0$ that 
    \begin{equation}
        \label{eq:consx2consz_org}
        \begin{aligned}
            \frac{\norm{\Pi\x_k}^2}{n} + \norm{\bar{x}_k - \proxp{\bar{z}_k}}^2 \leq \frac{1}{n(1-\gamma\rho)^2}\norm{\Pi\z_k}^2.
        \end{aligned}
    \end{equation}
\end{lemma}

\begin{proof}
    See Appendix \ref{app:pix2piz}.
\end{proof}


\textbf{Stationarity measure.} 
We use the multi-agent version of the ``proximal gradient mapping'' \cite{beck2017first}:
\begin{equation}
    \label{eq:prox-grad}
    \begin{array}{l}
        \frac{1}{n}\sum\limits_{i=1}^n\norm{\gamma^{-1}\Fnat{x_{i}}}^2.
    \end{array}
\end{equation}
Such a measure has been used in previous works including \cite{xiao2023one,mancino2023proximal}.

Lemma \ref{lem:stationarity} links {  measure~\eqref{eq:prox-grad}}\footnote{It is equivalent to use the proximal gradient mapping and the normal map to measure the stationarity, though discrepancies emerge at $\varepsilon$-stationarity points \cite[FIG. 2]{milzarek2023convergence}. Our results essentially establish $\sumn\normi{\Fnor{z_{i,k}}}^2/n<\varepsilon^2$ due to Lemma~\ref{lem:stationarity}.} to the normal map $\Fnor{\cdot}$.
\begin{lemma}
    \label{lem:stationarity}
    Let Assumptions \ref{as:smooth}, \ref{as:phi}, \ref{as:graph} hold. We have for any $k\geq 0$ that 
    \begin{equation}
        \label{eq:Fnat2Fnor}
            \frac{1-\gamma\rho}{n}\sum\limits_{i=1}^n\norm{\frac{\Fnat{x_{i,k}}}{\gamma}}^2
            \leq 2\norm{\Fnor{\bar{z}_k}}^2+ \frac{2\Ln^2}{n}\norm{\Pi\z_k}^2.
    \end{equation}
\end{lemma}

\begin{proof}
    See Appendix \ref{app:stationarity}.
\end{proof}

\subsection{A Multi-step Analysis}
\label{subsec:multi}

Lemma \ref{lem:stationarity} underscores the goal of our analysis: bounding $\normi{\Fnor{\bar{z}_k}}^2$ and $\Ln^2\normi{\Pi\z_k}^2/n$ simultaneously. However, analyzing the recursions of these terms between the $(k+1)$-th and the $k$-th iteration can be problematic, resulting in large accumulated errors.\footnote{This is part of the reasons that previous works require large mini-batches \cite{ghadimi2016mini} or variance reduction \cite{idrees2024analysis} to achieve convergence.}
To address this issue, we analyze the recursions between iterates $\crki{(x_{i,k_1},z_{i,k_1})}_{i=1}^n$ and $\crki{(x_{i,k_2},z_{i,k_2})}_{i=1}^n$ for integers $k_1$ and $k_2$ satisfying $0\leq k_1<k_2:= k_1 + m$ ($m\geq 1$). The intuition is, by choosing a sufficiently large $m$, the errors can be better controlled. { In particular, setting $m=1$ recovers the one-step analysis, and the choice of $m$ affects only the intermediate bounds on the stepsize and does not worsen the final stepsize in Theorem~\ref{thm:abc}.}

{ To facilitate the subsequent analysis, we present Table~\ref{tab:notations} that summarizes the key notations.} 
{
\begin{table}[htbp]
    \centering
    \renewcommand{\arraystretch}{1.2}
    {
\begin{tabular}{@{}cc@{}}
\toprule
Notation        & Definition                                                                                  \\ \midrule
$\delta_{i,k}$  & $g_{i}(x_{i,k};\xi_{i,k}) - \nabla f_i(x_{i,k})$                                            \\
$\Delta_{k}$    & $\g_k - \nabla F(\x_k)$                                                                     \\
$\Ln$           & $\frac{L - \rho + 2/\gamma}{1 - \gamma\rho}$                                          \\
$\Pi$           & $I - \1\1^{\T}/n$                                                                           \\
$D_1$      & $[V^{-1}]_l\hLambda_a \hU^{\T}$ \\
$D_2$      & $[V^{-1}]_r \hLambda_b^{-1}\hLambda_a\hU^{\T}$ \\
$\beta_{0:p}^2$ & $\sum_{t=0}^{p-1}\beta^{2t},\forall p\geq 1$                                                                \\
$\ccD_1^2$      & $\Ln^2\prti{\normi{V}^2 + 1} + 4 \C\prti{L\normi{V}^2 + 1/\gamma}$                          \\
$\ccD_2^2$      & $(\Ln^2 + 4\C L + 4\C\gamma^{-1})(\normi{D_1}^2 + 1)(\normi{V}^2 + 1)$       \\
$\cC_0$  & $\frac{3-4\gamma\rho}{2\prti{3 - 4\gamma\rho+4\gamma^2 L^2}}$   \\
$H_1$           & $ 8\alpha^2 m\ccD_1^2 \normi{D_1}^2 \beta_{0:m}^2 + \frac{8\alpha^2 m^2\ccD_1^2 + 1}{n}$         \\
$H_2$           & $\normi{D_1}^2 + \frac{4\alpha^2 m \normi{D_2}^2\Ln^2}{n(1-\beta^{2m})}$                         \\ 
$\cA(\alpha)$   & $\frac{133\alpha^3 m\normi{D_1}^2\ccD_1^{2}}{1-\beta^2}+\frac{400\alpha^5m \normi{D_2}^2 \ccD_2^2\Ln^2}{n(1-\beta^2)^2}$\\
\bottomrule
\end{tabular}}
\caption{ Notations used in the analysis. $D_1$ and $D_2$ are associated with the transformed consensus error; $\cC_0$ is the coefficient in the Lyapunov function; $\ccD_1$ and $\ccD_2$ are problem-dependent constants (depending on $L$, $\rho$, and $\C$); $H_1$, $H_2$, and $\cA(\alpha)$ are stepsize-dependent constants. 
}
\label{tab:notations}
\end{table}
}

The key strategy is to locate a Lyapunov function and construct its dynamic between the $k_1$-th and $k_2$-th iteration. Specifically, we consider $\cL_{k}$ in \eqref{eq:cLk1k2} below:
    \begin{align}
        \label{eq:cLk1k2}
            \cL_{k}:= \cH_{k} + \frac{{ 50\alpha \ccD_1^2}}{n(1-\beta^2)}\norm{\bu_{k}}^2,
    \end{align}
where $\beta$ is defined in Remark~\ref{rem:cons_beta} and is assumed to satisfy $\beta\in(0,1)$. The function $\cH_k$ is defined below.
\begin{align}\label{eq:cH_k1k2}
        \cH_k:= \psi(\proxp{\bar{z}_k}) - \bpsi + \frac{\gamma \cC_0}{2}\norm{\Fnor{\bar{z}_k}}^2.
\end{align}
Here, {$\cC_0$ is defined as in Table~\ref{tab:notations}, satisfying
$\cC_0 \in [4/9,1/2)$  given that $\gamma\in(0, 1/[5(\rho + L)]]$}, and $\bpsi:= f^* + \varphi^*$. The Lyapunov function $\cH_{k_1}$ is inspired by the works in \cite{milzarek2023convergence,ouyang2024trust}. 

We follow the roadmap below in presenting the remaining analysis of this section.
\begin{enumerate}
    \item Lemma~\ref{lem:m_step} rewrites the updates of the unified algorithmic framework in Lemma~\ref{lem:re} to accommodate the multi-step analysis.
    \item Lemma \ref{lem:cH_descent} characterizes the approximate descent property between $\cH_{k_2}$ and $\cH_{k_1}$, motivating us to upper bound an ``additional error'' term in the next step. 
    \item Lemma~\ref{lem:ek1k2} upper bounds the ``additional error'' term resulting from the stochastic gradients and consensus updates. 
    \item Derive the recursion for the transformed consensus error $\normi{\bu_{k_1}}^2$ in Lemma \ref{lem:norm-abc_cons_tw}. 
    \item Obtain the recursion for the Lyapunov function $\cL_{k_1}$ by combining the results derived above (Lemma~\ref{lem:cL}). 
\end{enumerate}

We begin by examining the updates between the $k_2$-th iteration and the $k_1$-th iteration for integers $k_1$ and $k_2$ satisfying $0\leq k_1<k_2:= k_1 + m$ ($m\geq 1$). In particular, the averaged iterate $\bar{z}_{k_2}$ can be regarded as performing $m$-step normal map updates on $\bar{z}_{k_1}$ with some additional error $e_{k_1:k_2}$ \eqref{eq:ek1k2_def}.

\begin{lemma}
    \label{lem:m_step}
    Let Assumptions \ref{as:smooth}, \ref{as:phi}, \ref{as:graph} and \ref{as:abcW} hold. We have for any $k_2= k_1 + m$ ($m\geq 1$) that
    \begin{equation}
        \label{eq:zbar_tw}
        \begin{aligned}
            \bar{z}_{k_2} 
            &= \bar{z}_{k_1} - m\alpha \Fnor{\bar{z}_{k_1}} + e_{k_1:k_2},\text{ and }
        \end{aligned}
    \end{equation}
    \begin{equation}
        \label{eq:cons_tw}
        \begin{aligned}
            &\bu_{k_2} = \Gamma^{k_2 -k_1} \bu_{k_1} - \alpha\sum_{p=k_1}^{k_2 - 1}{\Gamma^{k_2-1-p}D_1}\Delta_p\\
            &  - \alpha\sum_{p=k_1}^{k_2 - 1}{ \Gamma}^{k_2-1-p}
                { D_1}\brk{\Fnorb{\z_p} - \Fnorb{\barz_p} }\\
            & - \alpha\sum_{p=k_1}^{k_2 - 1}{ \Gamma}^{k_2-1-p} { D_2}\brk{\Fnorb{\barz_{p + 1}} - \Fnorb{\barz_p}},
        \end{aligned}
    \end{equation}
    where{$D_1$ and $D_2$ are defined in Table~\ref{tab:notations}\footnote{ We have $\normi{D_1}\sim\orderi{1}$ for both \normgt~and \normed, while $\normi{D_2}\sim\orderi{1/(1-\lambda)}$ for \normgt~and $\normi{D_2}\sim\orderi{1/\sqrt{1-\lambda}}$ for \normed.}, and}
     {
     \begin{equation}
        \label{eq:ek1k2_def}
        \begin{aligned}
            &e_{k_1:k_2} := \alpha \sum_{p=k_1}^{k_2-1}\brk{\Fnor{\bar{z}_{k_1}} - \Fnor{\bar{z}_p}}  - \frac{\alpha}{n}\sum_{p=k_1}^{k_2-1} \sumn\delta_{i,p}\\
            &\quad -\frac{\alpha}{n}\sum_{p=k_1}^{k_2-1} \sumn \brk{\Fnori{i}{z_{i,p}} - \Fnori{i}{\bar{z}_p}}.
        \end{aligned}
     \end{equation}
     }
    
\end{lemma}

\begin{proof}
    { Unrolling the updates in \eqref{eq:transformed} over $(k_2 - k_1)$ steps and utilizing the definitions of $D_1$ and $D_2$ in Table \ref{tab:notations} yields the desired results.}
\end{proof}

We next consider the approximate descent property of the Lyapunov function $\cH_k$, inspired by the works \cite{milzarek2023convergence,ouyang2024trust}.

\begin{lemma}
    \label{lem:cH_descent}
    Let Assumptions~\ref{as:smooth}-\ref{as:abcW} hold. Set $\alpha$ and $\gamma$ to satisfy $\alpha\leq \min\crki{\gamma/(2m), 1/(10 mL)}$ and {$\gamma \leq 1/[5(\rho + L)]$}. Then, for any $k_2= k_1 + m$ ($m\geq 1$),
    \begin{equation}
        \label{eq:cH_descent}
        \cH_{k_2}\leq \cH_{k_1} - \frac{\alpha m\cC_0}{2}\norm{\Fnor{\bar{z}_{k_1}}}^2 + \frac{1}{\alpha m}\norm{e_{k_1:k_2}}^2.
    \end{equation}
\end{lemma}

\begin{proof}
    See Appendix \ref{app:cH_descent}.
\end{proof}

Let the filtration $\crk{\cF_k}_{k\geq 0}$ be generated by $\crki{\xi_{i,\ell}| i\in\cN, \ell = 0,1,\ldots, k-1}$, the set of random variables (such as random data samples) utilized by the agents before the $k$-th iteration.
We next bound the term $\condEi{\normi{e_{k_1:k_2}}^2}{\cF_{k_1}}$, which arises due to the stochastic gradients and the consensus error among the networked agents.

\begin{lemma}
    \label{lem:ek1k2}
    Let Assumptions~\ref{as:smooth}-\ref{as:abcW} hold. Set $\alpha$ and $\gamma$ to satisfy {$\alpha\leq \min\crki{1/(16m\Ln\sqrt{\normi{D_1}^2+ 1}), \sqrt{1-\beta^2}/(4\normi{D_2} \Ln)}$, $\gamma\leq 1/[5(\C + L + \rho)]$.}
    Denote $\cM_{k_1:k_2}:=\sum_{p=k_1}^{k_2 - 1}\brki{\normi{\bar{z}_{k_1} - \bar{z}_p}^2 + \normi{\bu_p}^2/n}$. Then, for any $k_2= k_1 + m$ ($m\geq 1$),
    \begin{equation}
        \label{eq:Eek1k2_ub}
        \begin{aligned}
            &\condE{\norm{e_{k_1:k_2}}^2}{\cF_{k_1}} \leq \frac{24\alpha^2 m{\ccD_1^2}\beta_{0:m}^2}{n}\norm{\bu_{k_1}}^2\\
            &+ 3\alpha^2m H_1\crk{4\C\brk{\psi\prti{\proxp{\bar{z}_{k_1}}} - \bpsi} + \prt{2\C\sigfn + \sigma^2}}\\
            &+ { 6\alpha^2 m^2\brk{ \frac{\C\gamma}{mn} + \alpha^2m^2\ccD_1^2\prt{\normi{D_1}^2 + 2}}}\norm{\Fnor{\bar{z}_{k_1}}}^2,
        \end{aligned}
    \end{equation}
    and
    \small
    \begin{equation}
        \label{eq:sumk1k2_ub}
        \begin{aligned}
            &\condE{\cM_{k_1:k_2}}{\cF_{k_1}} \leq 2\alpha^2m^3\prt{ 2 + \normi{D_1}^2}\norm{\Fnor{\bar{z}_{k_1}}}^2\\
            &+ 8\alpha^2m \prt{\norm{D_1}^2\beta_{0:m}^2 + \frac{m}{n}}\crk{4\C \brk{\psi\prti{\proxp{\bar{z}_{k_1}}} - \bpsi} \right.\\
            &\left. + 2\C\sigfn + \sigma^2} + \frac{8\beta_{0:m}^2\normi{\bu_{k_1}}^2}{n},
        \end{aligned}
    \end{equation}\normalsize
    where{$H_1$, $\ccD_1^2$ and $\beta_{0:m}^2$ are defined in Table \ref{tab:notations}} and $\sigfn := f^* - \sumn f_i^*/n$.
    
\end{lemma}

\begin{proof}
See Appendix \ref{app:ek1k2}.
\end{proof}

Lemma \ref{lem:norm-abc_cons_tw} states the recursion of the transformed consensus error term.

\begin{lemma}
    \label{lem:norm-abc_cons_tw}
    Let Assumptions~\ref{as:smooth}-\ref{as:abcW} hold. Set $\alpha$ and $\gamma$ to satisfy {$\alpha\leq \min\crki{1/(16m\Ln\sqrt{\normi{D_1}^2+ 1}), \sqrt{1-\beta^2}/(4\normi{D_2} \Ln)}$, $\gamma\leq 1/[5(\C + L + \rho)]$.}
    Then, for any $k_2= k_1 + m$ ($m\geq 1$),
   \begin{equation}
    \label{eq:Euk_re}
     \begin{aligned}
         &\condE{\norm{\bu_{k_2}}^2}{\cF_{k_1}}\leq \frac{1 + \beta^{2m}}{2}\norm{\bu_{k_1}}^2 \\
         &+ 4\alpha^2 n\beta^2_{0:m}{ H_2}\norm{\Fnor{\bar{z}_{k_1}}}^2+ \frac{8\alpha^2m n\ccD_2^2}{1-\beta^{2m}}\condE{\cM_{k_1:k_2}}{\cF_{k_1}}\\
         &+ 2\alpha^2 n \beta_{0:m}^2{ H_2}\crki{4\C\brki{\psi\prti{{\proxp{\bar{z}_{k_1}}}} - \bpsi} + 2\C\sigfn + \sigma^2},
     \end{aligned}
   \end{equation}
   { where $H_2$ is defined in Table \ref{tab:notations}.}
    
\end{lemma}

\begin{proof}
     The proof { proceeds analogously to Step IV in the proof of Lemma~\ref{lem:ek1k2}}; we omit it here due to space limitations. Full details are provided in \cite[Appendix II-J]{huang2024distributed}.
\end{proof}

In Lemma \ref{lem:cL}, we obtain a critical relation related to the Lyapunov function $\cL_{k}$. It serves as a cornerstone for deriving the convergence guarantee of the considered algorithmic framework.

\begin{lemma}
    \label{lem:cL}
    Let Assumptions~\ref{as:smooth}-\ref{as:abcW} hold. Set {$\alpha$ and $\gamma$ to satisfy $\alpha\leq \min\{\nicefrac{1}{(80m\ccD_1\sqrt{\normi{D_1}^2 + 2})}, \nicefrac{(1-\beta^2)}{(80\ccD_1\sqrt{\normi{D_1}^2 + 2})},\\ 
    \nicefrac{\sqrt{(1-\beta^2)(1-\beta^{2m})}}{(80\sqrt{m} \ccD_2)}, \nicefrac{\sqrt{1-\beta^2}}{(80\normi{D_2}\ccD_1)}\}$ and $\gamma \leq 1/{[5(\rho + L + 18\C)]}$.}
    Then, we have for any $k_1\geq 0 \text{ and }k_2 = k_1 + m$ $(m\geq 1)$ that 
    {
    \begin{equation}
        \label{eq:cL}
        \begin{aligned}
            &\condE{\cL_{k_2}}{\cF_{k_1}}\leq \brk{1 + \frac{20\alpha\C}{n} + 4\C \cA(\alpha)}\cL_{k_1} \\
            & - \frac{\alpha m}{9}\norm{\Fnor{\bar{z}_{k_1}}}^2 + \brk{\frac{5\alpha }{n} + \cA(\alpha)}\prt{2\C\sigfn + \sigma^2},
        \end{aligned}
    \end{equation}
    where $\cA(\alpha)$ is a stepsize-dependent function that scales as $\orderi{\alpha^3}$; see Table~\ref{tab:notations} for its definition.
    }
\end{lemma}

\begin{proof}
    See Appendix \ref{app:cL}.
\end{proof}

\begin{remark}
    Suppose we consider a special case where $\C=0$ (bounded variance). Then, the recursion \eqref{eq:cL} can lead to a dominant error term $\orderi{\sigma^2/(mn)}$ due to the stochastic gradients when upper bounding the term $\sum_{{k_1}}\normi{\Fnor{\bar{z}_{k_1}}}^2$.
Thus, choosing $m = \orderi{\ceili{\sqrt{K/n}}}$ leads to a complexity of $\orderi{1/\sqrt{nK}}$ for a given iteration number $K>0$.
    The subsequent section formalizes such an argument for $\C\geq0$.
\end{remark}

\section{Convergence Results}
\label{sec:res}

In this section, we establish the iteration complexity for \normgt~and \normed~with respect to the iteration number $K>0$. Let $K = Tm + Q$, where $T \geq 0$ and $Q \in [0, m-1]$ are integers. We consider a subsequence $\crk{\tk_j}_{j\geq 0}$ of $\crki{0,1,\ldots, K-1}$. The subsequence is specified as follows and illustrated in Fig.~\ref{fig:tkj}.
\begin{align}
    \label{eq:tkj}
    \tk_0 = 0, \tk_1 = m, \tk_2 = 2m,\ldots, \tk_T = Tm.
\end{align}
\begin{figure}[ht]
    \centering
    \resizebox{0.48\textwidth}{!}{\begin{tikzpicture}
        \draw[thick,->] (0,0) -- (14,0) node[above left]{Iteration};
    
        \node at (0, 0) [below, yshift=-0.2cm] {$0$};
        \node at (1, 0) [below, yshift=-0.2cm] {$1$};
        \node at (2, 0) [below, yshift=-0.2cm] {$2$};
        \node at (4, 0) [below, yshift=-0.2cm] {$m$};
        \node at (5, 0) [below, yshift=-0.12cm] {$m+1$};
        \node at (7, 0) [below, yshift=-0.2cm] {$2m$};
        \node at (10, 0) [below, yshift=-0.2cm] {$Tm$};
        \node at (13, 0) [below, yshift=-0.2cm] {$K - 1$};
    
        \node at (0, -1) {$\tk_0$};
        \node at (4, -1) {$\tk_1$};
        \node at (7, -1) {$\tk_2$};
        \node at (10, -1) {$\tk_T$};
        \node at (13, -1) {$\tk_T + Q - 1$};
        
        \node at (3, 0) [below, yshift=-0.2cm] {$\cdots$};
        \node at (6, 0) [below, yshift=-0.2cm] {$\cdots$};
        \node at (8, 0) [below, yshift=-0.2cm] {$\cdots$};
        \node at (11, 0) [below, yshift=-0.2cm] {$\cdots$};
    
        \draw[thick] (-0.3, -0.2) -- (-0.3, -1.2) -- (0.3, -1.2) -- (0.3, -0.2) -- cycle;
        \draw[thick] (3.7, -0.2) -- (3.7, -1.2) -- (4.3, -1.2) -- (4.3, -0.2) -- cycle;
        \draw[thick] (6.7, -0.2) -- (6.7, -1.2) -- (7.3, -1.2) -- (7.3, -0.2) -- cycle;
        \draw[thick] (9.7, -0.2) -- (9.7, -1.2) -- (10.3, -1.2) -- (10.3, -0.2) -- cycle;
        \draw[thick] (12, -0.2) -- (12, -1.2) -- (13.9, -1.2) -- (13.9, -0.2) -- cycle;
    \end{tikzpicture}
    }
    \caption{An illustration of the subsequence $\crki{\tk_j}_{k\geq 0}$.}
    \label{fig:tkj}
\end{figure}

The goal is to upper bound $\sum_{k=0}^{K-1}\E\brki{\normi{\Fnor{\bar{z}_k}}^2}$ and $\sum_{k=0}^{K-1}\E\brki{\normi{\Pi\z_k}^2}$ according to Lemma \ref{lem:stationarity}. 
To achieve the goal, we utilize the multi-step analysis from Section~\ref{sec:ana}. Regarding $\sum_{k=0}^{K-1}\E\brki{\normi{\Fnor{\bar{z}_k}}^2}$, we divide the summation into two parts: one summing over the points $\tk_0,\tk_1,\ldots, \tk_T$, and the other summing over the interval $\tk_j, \tk_j + 1,\ldots,\tk_j + m - 1$ for each $\tk_j$ based on Fig.~\ref{fig:tkj}. Specifically, we have 
\begin{equation}
    \label{eq:sum_norm}
    \begin{aligned}
        &\sum_{k=0}^{K-1} \E\brk{\norm{\Fnor{\bar{z}_k}}^2} = \sum_{t=0}^{T-1}\sum_{q=0}^{m-1}\E\brk{\norm{\Fnor{\bar{z}_{\tk_t + q}}}^2}\\
        & + \sum_{q=0}^{Q-1}\E\brk{\norm{\Fnor{\bar{z}_{\tk_T + q}}}^2}\leq \sum_{t=0}^{T}\sum_{q=0}^{m-1}\E\brk{\norm{\Fnor{\bar{z}_{\tk_t + q}}}^2}.
    \end{aligned}
\end{equation}

A similar argument applies to the aggregated consensus error $\sum_{k=0}^{K-1}\E\brki{\normi{\Pi\z_k}^2}$. 

Lemma~\ref{lem:sum_dist} presents a way to upper bound the desired stationarity measure 
based on \eqref{eq:sum_norm}.

\begin{lemma}
    \label{lem:sum_dist}
    Let Assumptions \ref{as:smooth}, \ref{as:phi}, \ref{as:graph}, and \ref{as:abcW} hold. 
    We have 
    \small
    \begin{equation}
        \label{eq:sum_dist}
        \begin{aligned}
            &\frac{1}{nK}\sum_{k=0}^{K-1}\sumn\E\brk{\norm{\gamma^{-1}\Fnat{x_{i,k}}}^2}
            \leq \frac{4m\sum\limits_{t=0}^T \E\brki{\normi{\Fnor{\bar{z}_{\tk_t}}}^2}}{(1-\gamma\rho)^2K}\\
            &+ \frac{4\Ln^2\prti{\normi{V}^2 + 1}}{(1-\gamma\rho)^2K}\sum_{t=0}^T {\E\brk{\cM_{\tk_t: (\tk_{t+1})}}}.
        \end{aligned}
    \end{equation}\normalsize
\end{lemma}

\begin{proof}
    The proof is straightforward by noting $\normi{\Fnor{\bar{z}_{\tk_t + q}}}^2 \leq 2\Ln^2 \normi{\bar{z}_{\tk_t} - \bar{z}_{\tk_t + q}}^2 + 2\normi{\Fnor{\bar{z}_{\tk_t}}}^2$ and is omitted here. The details can be found in \cite[Appendix II-K]{huang2024distributed}.
\end{proof}

Lemma \ref{lem:cL_bounded} states that the terms in $\crki{ \E\brki{\cL_{\tk_j}}}_{j=0}^{T}$ are bounded for any given $T\geq 0$ if the stepsize $\alpha$ is chosen properly. Such a result helps decouple the right-hand-side of \eqref{eq:sum_dist}, which involves the summation of $\E\brki{\cL_{\tk_j}}$ for $j=0,1,\ldots,T$ due to~\eqref{eq:sumk1k2_ub}.

\begin{lemma}
    \label{lem:cL_bounded}
    Let Assumptions~\ref{as:smooth}-\ref{as:abcW} hold. Set $\alpha$ and $\gamma$ to satisfy{$ \gamma \leq \nicefrac{1}{[5(\rho + L + 18\C)]}$ and 
    \small
        \begin{align*}
            \alpha&\leq \min\crk{\frac{1}{80m\ccD_1\sqrt{\normi{D_1}^2 + 2}}, \frac{1-\beta^2}{80\ccD_1\sqrt{\normi{D_1}^2 + 2}},\frac{\sqrt{1-\beta^2}}{80\normi{D_2}\ccD_1}\right.\\
            &\left. \frac{\sqrt{(1-\beta^2)(1-\beta^{2m})}}{80\sqrt{m} \ccD_2}, \frac{mn}{96\C K}, \prt{\frac{1-\beta^2}{3192\normi{D_1}^2\ccD_1^{2}\C K}}^{\frac{1}{3}}, \right.\\
            &\left.\prt{\frac{n(1-\beta^2)^2}{9600 \normi{D_2}^2 \C\ccD_2^2\Ln^2K}}^{\frac{1}{5}}}.
        \end{align*}\normalsize
    }

    Then, we have for any $0<j\leq T + 1$ that 
    \begin{equation}
        \label{eq:cLkj_ub}
        \begin{array}{l}
            \E {\brki{\cL_{\tk_j}}}\leq \exp\prt{1}\brk{\cL_{0} + \frac{2\C\sigfn + \sigma^2}{4\C}}=:\hat{\cL}.
        \end{array}
    \end{equation}
    
\end{lemma}

\begin{proof}
    See Appendix \ref{app:cL_bounded}.
\end{proof}

Now we are ready to state the complexity result for the unified framework \eqref{eq:norm-abc} in Theorem \ref{thm:abc} below.

\begin{theorem}
    \label{thm:abc}
    Let Assumptions~\ref{as:smooth}-\ref{as:abcW} hold. Set {$\alpha$ and $\gamma$ to satisfy the conditions in Lemma \ref{lem:cL_bounded}}.
    Then,
    {
    \small
    \begin{equation}
        \label{eq:E_Prox-grad}
        \begin{aligned}
            &\frac{1}{nK}\sum_{k=0}^{K-1}\sumn \E\brk{\norm{\frac{\Fnat{x_{i,k}}}{\gamma}}^2} 
        \leq \frac{71\Delta_\psi}{\alpha K} + \frac{36\gamma \normi{\Fnor{\bar{z}_0}}^2}{\alpha K} \\
        &  + \frac{7032\normi{V^{-1}}^2\ccD_1^2\normi{\Pi\z_0}^2 + \sumn\normi{\Fnori{i}{\bar{z}_0}}^2}{n(1-\beta^2) K} + \frac{703\tilde{\sigma}^2}{mn} \\
        & + \frac{20200 \alpha^2 \ccD_2^2\tilde{\sigma}^2}{1-\beta^2}+ \frac{61200\alpha^4 \normi{D_2}^2\ccD_2^4\tilde{\sigma}^2}{n(1-\beta^2)^2},
        \end{aligned}
    \end{equation}\normalsize
    where $\Delta_\psi:= \psi(\bar{x}_0) - \bpsi$, $\tilde{\sigma}^2 := \exp(1)\prti{\C\cL_0 + 2\C\sigfn + \sigma^2}$, and $\cL_0 \leq \Delta_\psi + \gamma\normi{\Fnor{\bar{z}_0}}^2/2 + 100\alpha\normi{V^{-1}}^2\ccD_1^2 \normi{\Pi\z_0}^2/[n(1-\beta^2)] + \alpha \sumn\normi{\Fnori{i}{\bar{z}_0}}^2/[80(1-\beta^2)]$.
    }

    In addition, if we set {$m = \ceili{\sqrt{\nicefrac{K\tilde{\sigma}^2}{(n\ccD_2\Delta_\psi)}}+ \nicefrac{5}{(1-\beta^2)}}$,
    \small 
    \begin{equation}
        \label{eq:alpha_order}
        \begin{aligned}
            \alpha &= \frac{1}{96\sqrt{\frac{K\ccD_2\tilde{\sigma}^2}{n\Delta_\psi}} + \eta},\text{ and } \gamma = \frac{1}{5(\rho + L + 18\C)},
        \end{aligned}
    \end{equation}\normalsize
    with $\eta= [\nicefrac{9600\normi{D_2}^2\C \ccD_2^2\Ln^2 K}{[n(1-\beta^2)^2]}]^{1/5}+ \nicefrac{400\ccD_2}{(1-\beta^2)} +  \nicefrac{80\normi{D_2}\ccD_1}{\sqrt{1-\beta^2}}$ 
    $+ [\nicefrac{(3192\normi{D_1}^2\ccD_1^2\C K)}{(1-\beta^2)}]^{1/3}$ $+ \nicefrac{160\ccD_2[K\tilde{\sigma}^2/(n\ccD_2\Delta_\psi)]^{1/4}}{\sqrt{1-\beta^2}}$,
    }
    then 
    {
    \small
    \begin{equation}
        \label{eq:EFnor_order}
        \begin{aligned}
            &\frac{1}{nK}\sum_{k=0}^{K-1}\sumn\E\brk{\norm{\gamma^{-1}\Fnat{x_{i,k}}}^2}= \order{\sqrt{\frac{D_g\tilde{\sigma}^2}{nK}} \right.\\
            &\left. + \frac{D_g^{3/4}}{\sqrt{{1-\beta}}(nK^3)^{1/4}} + \frac{D_g^{2/3}}{[(1-\beta) K^2]^{1/3}} + \frac{\prt{\normi{D_2}D_g}^{2/5}}{[n(1-\beta)^2 K^4]^{1/5}} \right.\\
            &\left. + \frac{D_g}{(1-\beta)K} + \frac{\normi{D_2}(\normi{V} + 1)}{\sqrt{1-\beta}K} + \frac{\normi{V^{-1}}^2 \prti{\normi{V}^2 + 1}\normi{\Pi\z_0}^2 }{n(1-\beta)K}\right.\\
            &\left. + \frac{ \sumn\normi{\Fnori{i}{\bar{z}_0}}^2}{n(1-\beta)K}},
        \end{aligned}
    \end{equation}\normalsize
    where $D_g:= \sqrt{(\normi{D_1}^2 + 1)(\normi{V}^2 + 1)}$.
    }

\end{theorem}

\begin{proof}
    See Appendix \ref{app:abc}. 
\end{proof}

In what follows, we state the convergence guarantees for \normgt~and \normed. In particular, we not only demonstrate that the proposed algorithms enjoy $\orderi{1/\sqrt{nK}}$ rate of convergence when $K$ is large enough, but also establish the transient times \eqref{eq:def_kt} required for \normgt~and \normed~ to achieve convergence rates comparable to norM-CSGD, respectively. The proof is straightforward once we specify the choices of $A$, $B$, and $C$ in Theorem \ref{thm:abc}. { Both Corollaries \ref{cor:gt} and~\ref{cor:ed} assume that $\normi{\Pi\z_0}^2\sim\orderi{n^2}$ and $\sumn\normi{\Fnori{i}{\bar{z}_0}}^2\sim\orderi{n^2}$, which is mild in practice \cite{huang2024accelerated,nedic2018network}.}

\small
\begin{equation}
    \label{eq:def_kt}
    \begin{aligned}
        \KTN:=\inf_{K}&\crk{\frac{1}{nk}\sum_{t=0}^{k-1}\sumn\E\brk{\norm{\gamma^{-1}\Fnat{x_{i,t}}}^2}\right.\\
        &\left.\leq \order{\frac{1}{\sqrt{nk}}},\ \forall k\geq K}
    \end{aligned}
\end{equation}\normalsize

\begin{corollary}
    \label{cor:gt}
    Suppose the conditions in Theorem \ref{thm:abc} hold. Let $A = W$, $B = I-W$, and $C = W$. Then, $\normi{V}^2\leq 3,\; \normi{V^{-1}}^2\leq 9,\; \normi{D_1}^2\leq 9, \normi{D_2}^2 \leq {1}/{(1-\lambda)^2},\; \beta = \lambda.$ { In addition, $\tilde{\sigma}$ and $D_g$ are constants independent of $n$, $K$, and $1/(1-\lambda)$.} 
    The iterates generated by \normgt~satisfy 
    \small
    \begin{equation}
        \label{eq:EFnor_gt}
        \begin{aligned}
            &\frac{1}{nK}\sum_{k=0}^{K-1}\sumn\E\brk{\norm{\gamma^{-1}\Fnat{x_{i,k}}}^2}= \order{\frac{1}{\sqrt{nK}} \right.\\
            &\left.+ \frac{1}{\sqrt{{1-\lambda}}(nK^3)^{1/4}} + \frac{1}{[\prti{1-\lambda} K^2]^{1/3}} + {\frac{1}{(1-\lambda)^{3/2} K}}  \right.\\
            &\left.+  {\frac{1}{[n(1-\lambda)^4 K^4]^{1/5}}} + \frac{n}{(1-\lambda)K} }.
        \end{aligned}
    \end{equation}\normalsize

    The transient time of \normgt~is given by
    \begin{align*}
        \order{\max\crk{\frac{n^3}{(1-\lambda)^2}, {\frac{n}{(1-\lambda)^{3}}}}}.
    \end{align*}
\end{corollary}

\begin{remark}
    When the problem of interest is smooth, i.e., the term $\varphi(x)\equiv 0$, the transient time of DSGT behaves as $\orderi{\max\crki{{n^3}/{(1-\lambda)^2}, {n}/{(1-\lambda)^{8/3}} }}$ \cite{alghunaim2021unified}.
    For ring networks { where $(1-\lambda)\sim \orderi{n^{-2}}$, both \normgt\ and DSGT achieve the same transient time that behaves as $\orderi{n^7}$.} 
\end{remark}

Next, we show the same transient time for \normed~compared to its non-proximal counterpart \cite{huang2023cedas,alghunaim2021unified}, stated in Corollary \ref{cor:ed}.
\begin{corollary}
    \label{cor:ed}
    Suppose the conditions in Theorem \ref{thm:abc} hold and assume further that $W$ is positive semi-definite\footnote{We can choose $(I +W) /2$ as the mixing matrix that is positive semi-definite for any symmetric doubly stochastic matrix $W$. { For this case, $\underline{\lambda}\in(0.5,1)$.}}. Let $A = W$, $B = \prti{I-W}^{1/2}$, and $C = I$. Then, $\normi{V}^2\leq 4,\; \normi{V^{-1}}^2\leq {2}/{\underline{\lambda}},\; \normi{D_1}^2\leq 1, \normi{D_2}^2 \leq {1}/\prti{1-\lambda},\; \beta = \sqrt{\lambda},$
    where $\underline{\lambda}$ is the minimum non-zero eigenvalue of $W$. { In addition, $\tilde{\sigma}$ and $D_g$ are constants independent of $n$, $K$, and $1/(1-\lambda)$.}

    The iterates generated by \normed~satisfy 
    \small
    \begin{equation}
        \label{eq:EFnor_ed}
        \begin{aligned}
            &\frac{1}{nK}\sum_{k=0}^{K-1}\sumn\E\brk{\norm{\gamma^{-1}\Fnat{x_{i,k}}}^2}= \order{\frac{1}{\sqrt{nK}}  \right.\\
            &\left.+ \frac{1}{\sqrt{{1-\lambda}}(nK^3)^{1/4}} + \frac{1}{[\prti{1-\lambda} K^2]^{1/3}}  \right.\\
            &\left. + {\frac{1}{[n(1-\lambda)^3 K^4]^{1/5}}}+ \frac{n}{(1-\lambda)K} }.
        \end{aligned}
    \end{equation}\normalsize

    The transient time of \normed~is thus given by
    \begin{align*}
        \order{\frac{n^3}{(1-\lambda)^2}}.
    \end{align*}
\end{corollary}

\section{Numerical Experiments}
\label{sec:sims}


This section presents { three} numerical examples that illustrate the superior performance of \normed~and \normgt~in comparison to existing methods over ring graphs { and Erd\"os-R\'enyi graphs}. All methods utilize the same stochastic oracle to output stochastic gradients and apply the proximal operator once per iteration. Regarding the data distribution, we consider the heterogeneous setting, i.e., the samples are first sorted by their labels and then partitioned among the agents. The stationarity measure for all figures is $\sumn\E\brki{\normi{\gamma^{-1}\Fnat{x_{i,k}} }^2}/n$ for the distributed methods and $\E\brki{\normi{\gamma^{-1} \Fnat{x_k}}^2}$ for centralized methods.

\subsection{Vision-based Steering Control}
\label{subsec:steering}

{
We consider a distributed constrained optimization problem motivated by vision-based steering control \cite{li2015vision}. We use Dataset~1 in SullyChen driving datasets\footnote{https://github.com/SullyChen/driving-datasets}. 
Each image is converted to grayscale, downsampled to $20\times 36$, and vectorized into a feature vector in $\R^{720}$. 
The resulting optimization problem is
\begin{equation}
    \label{eq:steering}
    \min_{x\in\R^p} \psi(x)=
    \frac{1}{n}\sumn
    \frac{1}{|\mathcal{S}_i|}
    \sum_{j=1}^{|\mathcal{S}_i|}
    \frac{1}{2}\big(a_{ij}^{\T}x - b_{ij}\big)^2
    + \delta_{\mathcal{X}}(x),
\end{equation}
where $a_{ij}\in\R^p$ denotes the processed image at agent $i$, $b_{ij}\in\R$ is the normalized steering angle, and $\delta_{\mathcal{X}}$ is the indicator of the box constrained set $\mathcal{X}:=\{x\in\R^p:x \in[-C,C]^p\}$ with $C = 0.005$.
The box constraint reflects the physical requirement that the control input be bounded \cite{li2015vision}. 
 
We compare \normed~and \normgt~with Prox-DASA, Prox-DASA-GT in \cite{xiao2023one}, as well as the benchmark methods norM-CSGD and Prox-CSGD, in Fig.~\ref{fig:steering}. Both \normed~and \normgt~converge faster than Prox-DASA and Prox-DASA-GT, and achieve accuracy closer to that of the centralized methods, demonstrating the advantage of the proposed approaches. In particular, \normed~achieves much faster convergence on the sparse ring graph (corresponding to a smaller $1-\lambda$), demonstrating its superior transient performance in practice.

\begin{figure}[htbp]
    \centering
    \subfloat[Ring graph, $n=50$, $1-\lambda = 2.6\times 10^{-3}$.]{\includegraphics[width=0.24\textwidth]{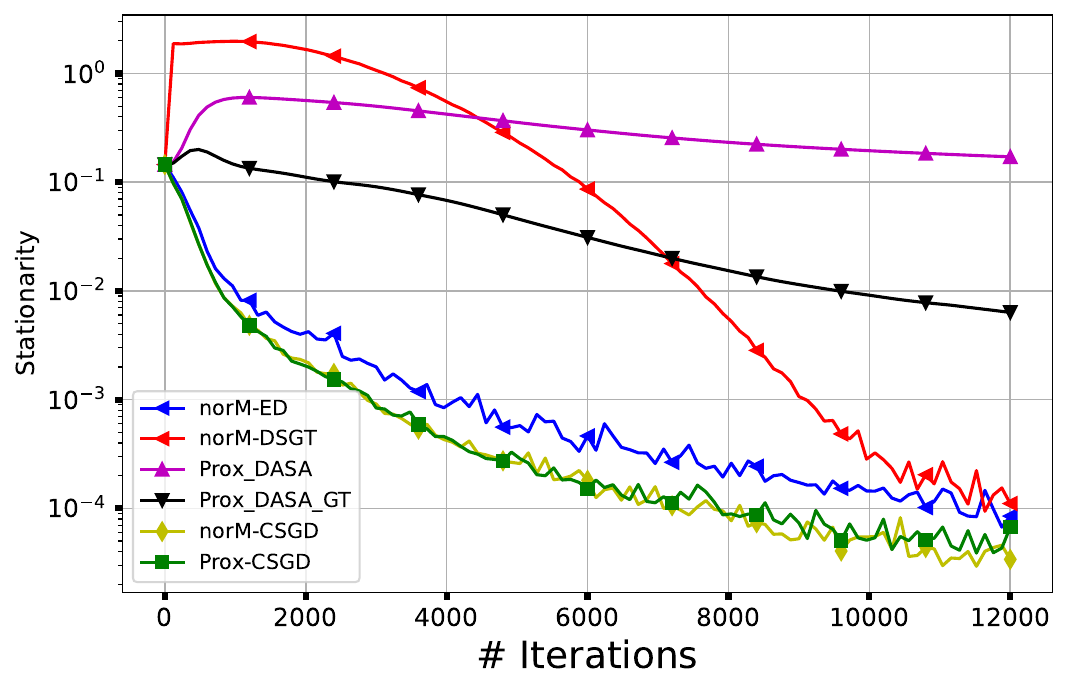}\label{fig:steering_ring50}}
    \subfloat[{ Erd\"os-R\'enyi graph, $n=50$, $1-\lambda = 0.1$.}]{\includegraphics[width=0.24\textwidth]{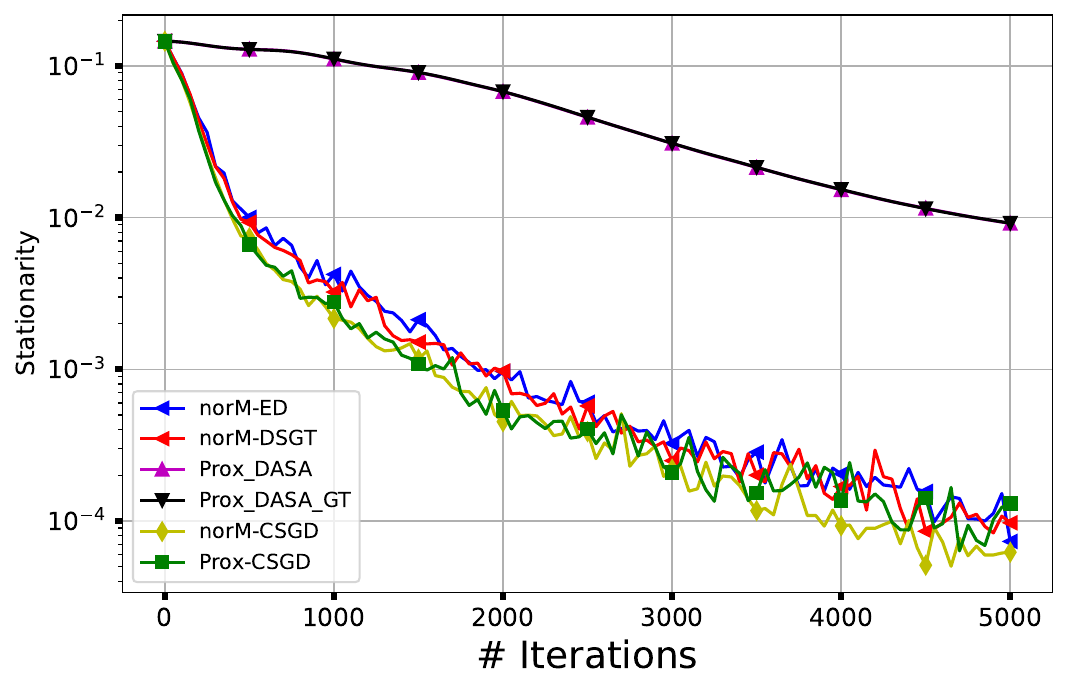}\label{fig:steering_er50}}
    \caption{ Comparison of the proposed methods and baselines for solving Problem \eqref{eq:steering} on the SullyChen dataset. For all methods, the stepsize is set to $0.0005$ on the ring graph and $0.001$ on the Erd\"os-R\'enyi graph ($\text{prob}=0.1$), the batch size is fixed at $32$, and the proximal parameter $\gamma$ is set to $0.1$.
    The results are averaged over $10$ independent runs.}
    \label{fig:steering}
\end{figure}

}
\subsection{Sparse Binary Classification}
\label{subsec:sparse}

We evaluate the algorithms on a sparse nonconvex binary classification problem \eqref{eq:sparse} following that in \cite{milzarek2019stochastic,milzarek2023convergence}. 
The experiment uses the MNIST dataset \cite{mnist}, extracting handwritten digits $2$ and $6$ as the training samples.
The optimization problem is formulated as follows:

\begin{equation}
    \label{eq:sparse}
        \min_{x\in\R^p} \psi(x) = \frac{1}{n}\sumn \frac{1}{|\mathcal{S}_i|}\sum_{j=1}^{|\mathcal{S}_i|} \brk{1 - \mathrm{tanh}\prt{b_{ij}\cdot a_{ij}^{\T}x}} + \nu\norm{x}_1,
\end{equation}
where $\cS_i$ is the local dataset of agent $i$, $a_{ij}\in\R^p$ is the $j$-th training sample, and $b_{ij}\in\crki{-1, 1}$ is the associated label of agent $i$. The hyperbolic tangent function, $\mathrm{tanh}(x)$, is defined as $\mathrm{tanh}(x):= [\mathrm{e}^x - \mathrm{e}^{-x}]/[\mathrm{e}^x + \mathrm{e}^{-x}]$. We set $\nu = 0.01$ in the simulations. 

We compare { the aforementioned algorithms}
in Fig.~\ref{fig:sparse} to highlight the enhanced performance of the proposed methods. 
 It can be seen that both \normed\footnote{ When the stepsize changes, we reset the iteration counter $k$ to zero in \normed. The total number of iterations is kept the same as in the other algorithms.} and \normgt~outperform Prox-DASA and Prox-DASA-GT, indicating their enhanced transient times. Notably, \normed~shows better performance than \normgt~on the ring graph. 

\begin{figure}[htbp]
    \centering
    \subfloat[Ring graph, $n=50$, $1-\lambda = 2.6\times 10^{-3}$.]{\includegraphics[width=0.24\textwidth]{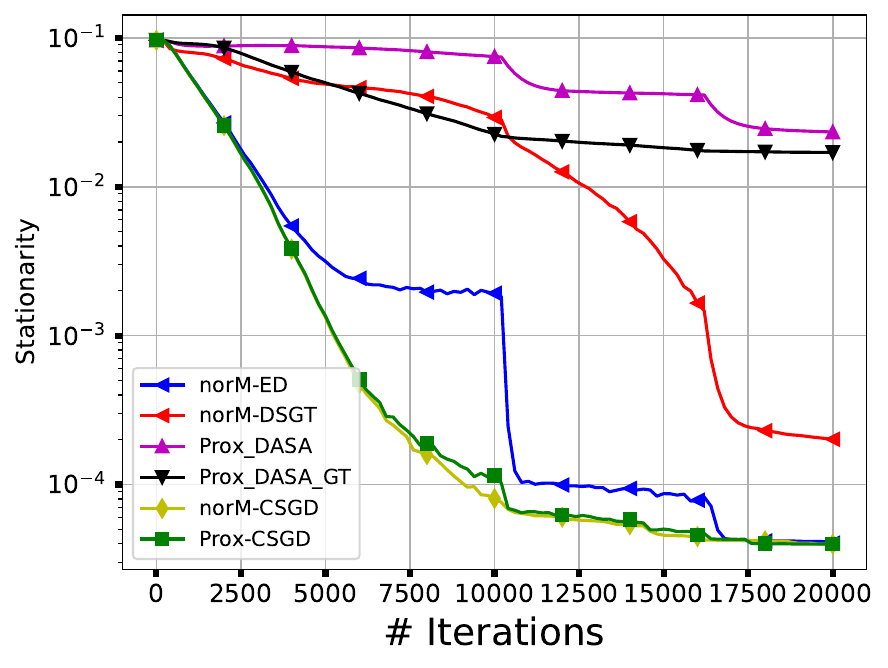}\label{fig:sparse_ring}}
    \subfloat[{ Erd\"os-R\'enyi graph, $n=50$, $1-\lambda = 0.12$.}]{\includegraphics[width=0.24\textwidth]{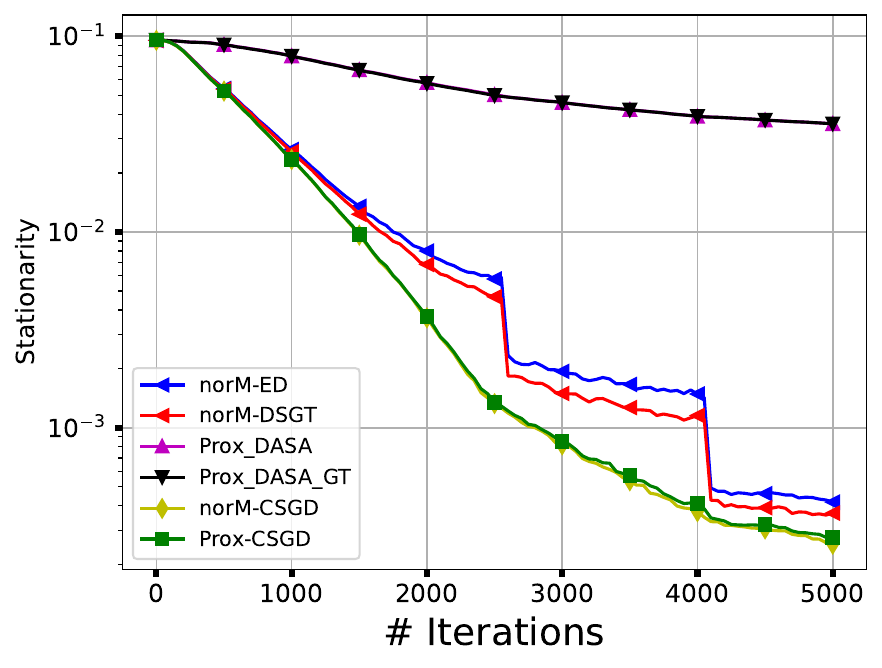}\label{fig:sparse_er}}
    \caption{ Comparison of the proposed methods and baselines for solving Problem \eqref{eq:sparse} on the MNIST dataset. For all the methods, the stepsizes are sequentially set to $1/20$, $1/100$, and $1/500$ on the ring graph, and $1/10$, $1/20$, and $1/40$ on the Erd\"os-R\'enyi graph ($\text{prob}=0.3$), and the proximal parameter $\gamma$ is set to $0.2$. 
    The results are averaged over $10$ independent runs.}
    \label{fig:sparse}
\end{figure}

\subsection{Neural Network}
\label{subsec:nn}

We further evaluate the aforementioned algorithms on a multi-class image classification task using the MNIST dataset.
The objective is to train a one-hidden-layer neural network with the sigmoid activation function and the elastic network \cite{zou2005regularization} regularizer $\varphi(x) = \nu_1 \normi{x}_1 + \nu_2\normi{x}_2^2$. Here, the parameters $\nu_1 = 0.001$ and $\nu_2 = 0.005$. Such a problem aligns with the structure of Problem \eqref{eq:P}.

Fig.~\ref{fig:nn} highlights the enhanced performance of \normed~and \normgt~compared to previous works. In particular, the proposed methods outperform the Prox-SGD type distributed optimization methods and demonstrate comparable convergence to centralized algorithms. 

\begin{figure}[htbp]
    \centering
    \subfloat[Ring graph, $n=30$, $1-\lambda = 7.3\times 10^{-3}$.]{\includegraphics[width=0.24\textwidth]{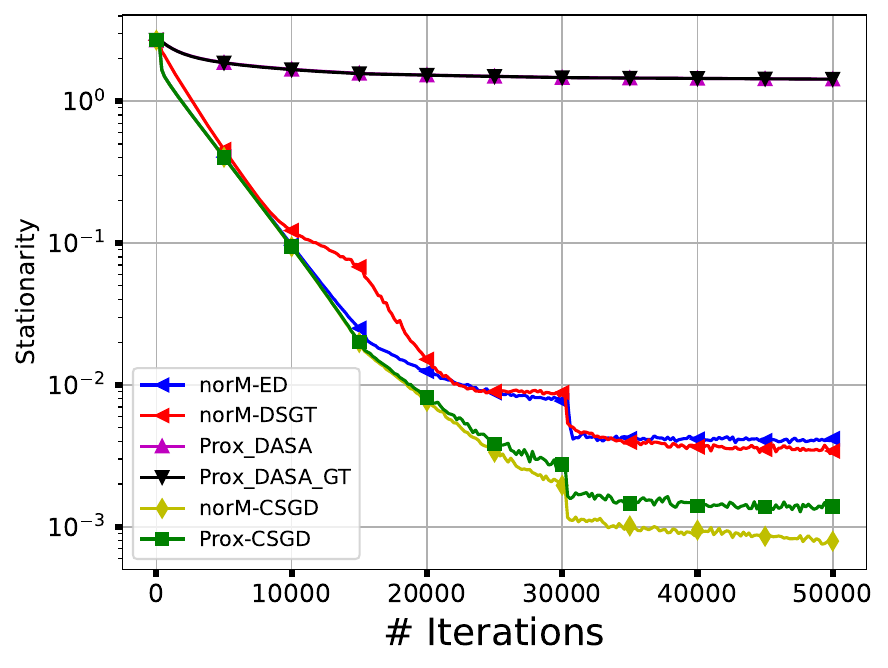}\label{fig:nn_n30}}
    \subfloat[{ Erd\"os-R\'enyi graph, $n=30$, $1-\lambda = 0.16$.}]{\includegraphics[width=0.24\textwidth]{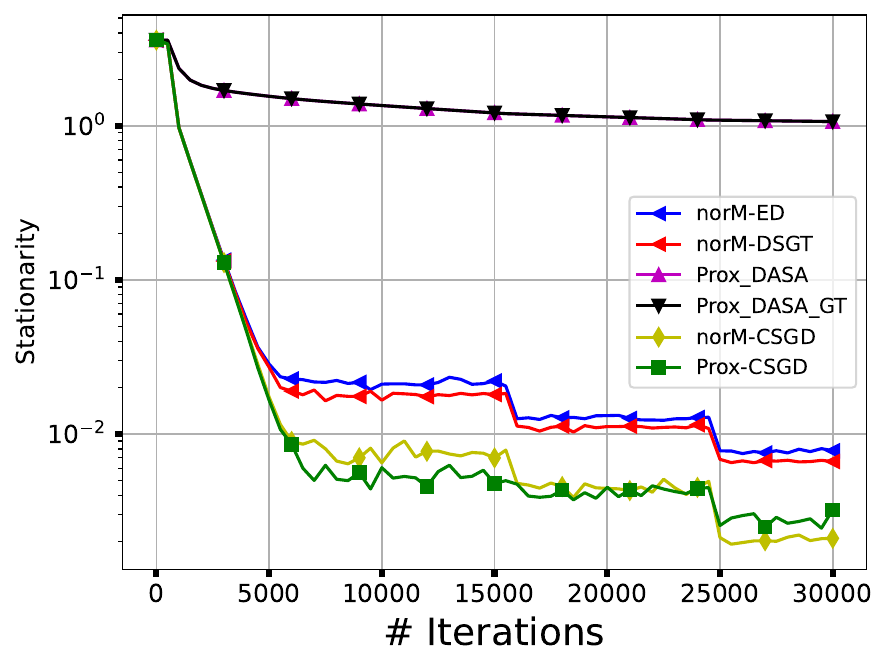}\label{fig:nn_er30}}
    \caption{ Comparison of the proposed methods and baselines for training a neural network on the MNIST dataset. For all the methods, the stepsizes are sequentially set  to $1/70$, $1/140$, and $1/400$ on the ring graph, and $1/10$, $1/20$, and $1/40$ on the Erd\"os-R\'enyi graph ($\text{prob}=0.3$), and the proximal parameter $\gamma$ is set to $0.02$. The results are averaged over $5$ independent runs.}
    \label{fig:nn}
\end{figure}

\section{Conclusions}
\label{sec:con}

{ This paper introduces} a unified algorithmic framework, norM-SABC-2, leveraging the normal map update scheme { to solve the distributed stochastic composite optimization problem over networks}. Within this framework, we propose \normgt~and \normed, both achieving network independent convergence rates 
and demonstrating superior transient times under a general variance condition.
In particular, \normed~matches the transient time of the non-proximal ED algorithm. Such a result is state-of-the-art for solving the considered problem to our knowledge. The \normgt~method shortens the transient time compared to previous distributed proximal gradient tracking based methods. 
Additionally, the employed multi-step analysis is of independent interest which may inspire further advancements in distributed optimization algorithms, particularly when utilizing the random reshuffling strategy for sampling stochastic gradients and dealing with time-varying networks. { Extending the framework to agent-specific regularizers $\varphi_i$ remains an interesting direction for future work.}

\section*{Acknowledgements}
We would like to thank Andre Milzarek from The Chinese University of Hong Kong, Shenzhen (CUHK-Shenzhen) and Junwen Qiu from the National University of Singapore for providing insightful discussions about normal map. 

\begin{appendices}
    \section{norM-CSGD}
    This part presents norM-CSGD (Algorithm~\ref{alg:norm-csgd}).
    \begin{algorithm}[H]
        \begin{algorithmic}[1]
            \STATE Set $z_{0}$, $\alpha$, $\gamma$, and $x_{0} = \proxpi{z_{0}}$ for the server.
            \FOR{$k=0, 1, 2, ..., K-1$}
            \STATE Server sends $x_k$ to all the agent.
            \FOR{Agent $i = 1, 2, ..., n$ in parallel}
            \STATE Acquires $g_{{i}, k} = g_i(x_{k};\xi_{i, k})$ and sends it to the server.\label{line:norm-cspgd_g}
            \ENDFOR
            \STATE Server aggregates $\bar{g}_k = \sumn g_{i,k}/n$ and updates $z_{k + 1} = z_{k} - \alpha\brki{\bar{g}_k + \gamma^{-1}\prti{z_{k} - x_{k}}}$, $x_{k + 1} = \proxpi{z_{k + 1}}$. \label{line:norm-cspgd_zx}
            \ENDFOR
        \end{algorithmic}
        \caption{A Normal Map-based Centralized Stochastic Gradient Descent (norM-CSGD)}
        \label{alg:norm-csgd}
    \end{algorithm}
    \section{Proofs}

    \subsection{Proof of Lemma \ref{lem:re}}
    \label{app:re}
    { By the definition of $\d_k'$}, the update \eqref{eq:norm-abc} can be rewritten as $\z_{k + 1} = \prt{AC - B^2}\z_k - \alpha A\brki{\Fnorb{\z_k} - \Fnorb{\barz_k} + \g_k - \nabla F(\x_k)}- \d'_k$, $\x_{k + 1} = \bproxp{\z_{k + 1}}$, and $\d'_{k + 1} = \d'_k + B^2 \z_k + \alpha A\brki{\Fnorb{\barz_{k + 1}} - \Fnorb{\barz_k}}$,
    { It follows from Assumption \ref{as:abcW} that 
    \small
    \begin{align*}
        &\begin{pmatrix}
            \hU^{\T} \z_{k + 1}\\
            \hLambda_b^{-1} \hU^{\T} \d'_{k + 1}
        \end{pmatrix} = \begin{pmatrix}
            \hLambda_a\hLambda_c - \hLambda_b^2 & - \hLambda_b \\
            \hLambda_b & I_{n-1}
        \end{pmatrix}\begin{pmatrix}
            \hU^{\T} \z_k\\
            \hLambda_b^{-1} \hU^{\T} \d'_k
        \end{pmatrix} \\
        &\quad -\alpha \begin{pmatrix}
            \hLambda_a \hU^{\T}\brki{\Fnorb{\z_k} - \Fnorb{\barz_k} + \Delta_k}\\
            \hLambda_b^{-1} \hLambda_a \hU^{\T}\brki{\Fnorb{\barz_{k + 1}} - \Fnorb{\barz_k}}
        \end{pmatrix}.
    \end{align*}\normalsize
    Multiplying both sides by $V^{-1}$ yields the desired result~\eqref{eq:buk}. Relation \eqref{eq:abc_zbark} follows from the fact that $\Fnor{\bar{z}_k} = \sumn\Fnori{i}{\bar{z}_k}/n$.}
    
    \subsection{Proof of Lemma \ref{lem:Fnor_Lip}}
    \label{app:Fnor_Lip}
    
    We first introduce Lemma \ref{lem:prox_wcvx} (\cite[(8)]{li2023new}), which states the nonexpansive property of $\proxp{\cdot}$ under Assumption \ref{as:phi}.
        \begin{lemma}
            \label{lem:prox_wcvx}
            Let Assumption \ref{as:phi} hold. Set $\gamma\in (0, \rho^{-1})$. We have for all $w,v\in\R^p$ that 
            \begin{enumerate}
                \item $\inproi{w-v, \proxp{w} - \proxp{v}}\geq (1-\gamma\rho)\normi{\proxp{w} - \proxp{v}}^2$, and 
                \item $\normi{w-v}\geq (1-\gamma\rho)\normi{\proxp{w} - \proxp{v}}$.
            \end{enumerate}
        \end{lemma}
        
    Based on Lemma \ref{lem:prox_wcvx}, we have for any $z, z'\in\R^p$ and $L$-smooth $f$ that $\normi{\Fnor{z} - \Fnor{z'}}\leq \gamma^{-1}\normi{\proxp{z} - \proxp{z'}}+\gamma^{-1}\normi{z - z'} +  \normi{\nabla f\prt{\proxp{z}} - \nabla f\prt{\proxp{z'}}}\leq \brki{{\prti{L + 1/\gamma}}/\prti{1-\gamma \rho} + \gamma^{-1}}\normi{z - z'},$
    where we invoked Lemma \ref{lem:prox_wcvx} in the last inequality. Since $f_i$ is $L$-smooth under Assumption \ref{as:smooth}, the above argument also applies to $\Fnori{i}{\cdot}$.
    
    
    
    \subsection{Proof of Lemma \ref{lem:pix2piz}}
    \label{app:pix2piz}
    Recall that $x_{i,k} = \proxp{z_{i,k}}$ for any $k\geq 0$. We have $\normi{\Pi\x_k}^2/n 
                =\sumn\normi{\proxp{z_{i,k}} - \proxp{\bar{z}_k} }^2/n
                - \normi{\proxp{\bar{z}_k} - \sumn\proxp{z_{i,k}}/n}^2
                \leq \normi{\Pi\z_k}^2 / \brki{n(1-\gamma\rho)^2} - \normi{\bar{x}_k - \proxp{\bar{z}_k}}^2,$
        where we invoked Lemma \ref{lem:prox_wcvx} in the last inequality. 

    \subsection{Proof of Lemma \ref{lem:stationarity}}
    \label{app:stationarity}
    
        We denote 
        $\LA := 1 + \prti{1 + \gamma L}/\prti{1-\gamma\rho} = \gamma\Ln$. Then, $\Fnat{\cdot}$ is $\LA$-Lipschitz continuous under Assumptions~\ref{as:smooth}~and~\ref{as:phi}: $\normi{\Fnat{x} - \Fnat{x'}} \leq \normi{x - x'} + \normi{x - x' - \gamma\prt{\nabla f(x) - \nabla f(x')}}/\prti{1-\gamma\rho} \leq \LA\normi{x - x'}.$
         
        Then, relation \eqref{eq:prox-grad} becomes $\sumn\normi{\gamma^{-1}\Fnat{x_{i,k}}}^2/n \leq 2\normi{\Fnor{\bar{z}_k}}^2/(1-\gamma\rho)^2 +{2\LA^2}\sumn\normi{x_{i,k} - \proxp{\bar{z}_k}}^2/\prti{n\gamma^2}$ for any $\bar{z}_k\in\R^p$,
        where the inequality holds by noting that $\Fnat{\proxp{\bar{z}_k}} = \proxp{\bar{z}_k} - \proxp{\bar{z}_k - \gamma\Fnor{\bar{z}_k}}.$
    
        We also have from Lemma \ref{lem:pix2piz} that $\sumn\normi{x_{i,k} - \proxp{\bar{z}_k}}^2/n = \normi{\Pi\x_k}^2/n + \normi{\bar{x}_k - \proxp{\bar{z}_k}}^2\leq \normi{\Pi\z_k}^2/\prti{n(1-\gamma\rho)^2}$.
        Combining the above inequalities leads to the desired result.

    \subsection{Proof of Lemma \ref{lem:cH_descent}}
    \label{app:cH_descent}

    The proof proceeds in three steps and is inspired by those in \cite{milzarek2023convergence,li2023new,ouyang2024trust}. For simplicity, we denote $\ux_k := \proxp{\bar{z}_k}$ for $k \geq 0$. 
    
    In the first two steps, we derive the recursion for $\psi(\ux_k)$ and $\normi{\Fnor{\bar{z}_k}}^2$, respectively. The last step constructs the Lyapunov function $\cH_k$ based on these results.
    
    \textbf{Step I: Obtaining the recursion for $\psi(\ux_{k_2})$}. 
        Due to Assumption \ref{as:phi} that $\varphi$ is $\rho$-weakly convex, we have for any $x, x'\in\dom{\varphi}$ and any $v\in\partial\varphi(x')$ that $\varphi(x)\geq \varphi(x') + \inproi{v, x - x'} - \rho\normi{x - x'}^2/2.$
    
        Note that $\gamma^{-1}\prti{\bar{z}_{k_2} - \ux_{k_2}}\in\partial \varphi(\ux_{k_2})$.
        For any $k_2 > k_1\geq 0$, setting $x' = \ux_{k_2}$, $x = \ux_{k_1}$, and $v = \gamma^{-1}\prti{\bar{z}_{k_2} - \ux_{k_2}}$ in the above inequality leads to $\varphi\prti{\ux_{k_2}} - \varphi\prti{\ux_{k_1}}\leq -\gamma^{-1}\inproi{\bar{z}_{k_2} - \ux_{k_2}, \ux_{k_1} - \ux_{k_2}} + \rho\normi{\ux_{k_1} - \ux_{k_2}}^2/2.$
    
        Based on $\Fnor{\bar{z}_{k_1}} = \nabla f(\ux_{k_1}) + \gamma^{-1}\prti{\bar{z}_{k_1} - \ux_{k_1}}$ and Assumption \ref{as:smooth} that $f$ is $L$-smooth, we have $f\prti{\ux_{k_2}} \leq f\prti{\ux_{k_1}} + \inproi{\Fnor{\bar{z}_{k_1}}, \ux_{k_2} - \ux_{k_1}} - \gamma^{-1}\inproi{\bar{z}_{k_1} - \ux_{k_1}, \ux_{k_2} - \ux_{k_1}}+ L\normi{\ux_{k_2} - \ux_{k_1}}^2/2.$
    
        Combining the above inequalities about $f(\ux_{k_2})$ and $\varphi(\ux_{k_2})$ yields $\psi\prti{\ux_{k_2}} - \psi\prti{\ux_{k_1}} \leq \prti{\prti{\rho + L}/{2} - \gamma^{-1}}\normi{\ux_{k_1} - \ux_{k_2}}^2 + \inproi{\Fnor{\bar{z}_{k_1}}, \ux_{k_{2}} - \ux_{k_1}} + \gamma^{-1}\inproi{\bar{z}_{k_2} - \bar{z}_{k_1}, \ux_{k_2} - \ux_{k_1}}$.
    
        \textbf{Step II: Deriving the recursion for $\normi{\Fnor{\bar{z}_{k_2}}}^2$}. 
        According to the definition of $\Fnor{z}$ in \eqref{eq:Fnor} and the update \eqref{eq:zbar_tw}, we have $\Fnor{\bar{z}_{k_2}} = \prti{1 - {m\alpha}{\gamma^{-1}}}\Fnor{\bar{z}_{k_1}} + \nabla f(\ux_{k_2}) - \nabla f(\ux_{k_1}) - \gamma^{-1}\prti{\ux_{k_2} - \ux_{k_1}} + \gamma^{-1}e_{k_1:k_2}.$ 
        Then, 
        \small
        \begin{equation}
            \label{eq:Fnor_zbar_tw_norm}
            \begin{aligned}
                &\norm{\Fnor{\bar{z}_{k_2}}}^2 
                \leq \prt{1 - \frac{\alpha m}{\gamma}}^2\norm{\Fnor{\bar{z}_{k_1}}}^2 + \gamma^{-2}\norm{e_{k_1:k_2}}^2 \\
                &+ 2\prt{1 - \frac{\alpha m}{\gamma}}\inpro{\Fnor{\bar{z}_{k_1}}, \nabla f\prti{\ux_{k_2}} - \nabla f\prti{\ux_{k_1}} \right.\\
                &\left.- \gamma^{-1}(\ux_{k_2} - \ux_{k_1} -e_{k_1:k_2})} + \prt{L + \gamma^{-1}}^2\norm{\ux_{k_2} - \ux_{k_1}}^2 \\
                &+ \frac{2}{\gamma^2}\inpro{\gamma\brk{\nabla f\prti{\ux_{k_2}} - \nabla f\prti{\ux_{k_1}}} - \prti{\ux_{k_2} - \ux_{k_1}},e_{k_1:k_2}},
            \end{aligned}
        \end{equation}\normalsize
        where we applied the Cauchy-Schwarz inequality and invoked the $L$-smoothness of $f$.
    
        \textbf{Step III: Obtaining the recursion for $\cH_{k_2}$.} 
        Substituting the derived inequalities in \textbf{Step I} and \textbf{Step II} into \eqref{eq:cH_k1k2} and rearranging it yield
        \small
        \begin{equation}
            \label{eq:cH_k1k2_can0}
            \begin{aligned}
                &\cH_{k_2} 
                \leq \brk{\prt{\frac{\rho + L}{2} - \frac{1}{\gamma}} + \frac{\cC_0\gamma}{2}\prt{ L + \frac{1}{\gamma}}^2}\norm{\ux_{k_1} - \ux_{k_2}}^2 \\
                &+ \psi(\ux_{k_1}) - \bpsi  + \frac{\cC_0}{2\gamma}\norm{e_{k_1:k_2}}^2  + \frac{\gamma \cC_0}{2}\prt{1 - \frac{\alpha m}{\gamma}}^2\norm{\Fnor{\bar{z}_{k_1}}}^2\\
                & + \frac{\cC_0}{\gamma}\inpro{\gamma\brk{\nabla f\prti{\ux_{k_2}} - \nabla f\prti{\ux_{k_1}}},e_{k_1:k_2}} + \inpro{\tc, \ux_{k_2} - \ux_{k_1}}\\
                &+ \cC_0\prt{1 - \frac{\alpha m}{\gamma}}\inpro{\Fnor{\bar{z}_{k_1}}, \gamma\brki{\nabla f\prti{\ux_{k_2}} - \nabla f\prti{\ux_{k_1}}} + e_{k_1:k_2}},
            \end{aligned}
        \end{equation}\normalsize
        where $\tc:= \Fnor{\bar{z}_{k_1}} + \gamma^{-1}\prti{\bar{z}_{k_2} - \bar{z}_{k_1}} - \cC_0\prti{1 - {\alpha m}/{\gamma}}\Fnor{\bar{z}_{k_1}} - \gamma^{-1}\cC_0e_{k_1:k_2}$. 
        According to \eqref{eq:zbar_tw}, we have $e_{k_1:k_2} - \alpha m \Fnor{\bar{z}_{k_1}} = \bar{z}_{k_2} - \bar{z}_{k_1}$ and $\Fnor{\bar{z}_{k_1}} = \prti{\bar{z}_{k_1} - \bar{z}_{k_2} + e_{k_1:k_2}}/(\alpha m)$. Then, 
        \begin{equation}
            \label{eq:Fnor2zbar_tw}
            \begin{aligned}
                &\tc= \prt{1 - \cC_0}\Fnor{\bar{z}_{k_1}} + \gamma^{-1}\prt{\bar{z}_{k_2} - \bar{z}_{k_1}} \\
                &- {\cC_0}\gamma^{-1}\brk{e_{k_1:k_2} - \alpha m\Fnor{\bar{z}_{k_1}}}\\
                &= \prt{1 - \cC_0}\prt{\frac{1}{\gamma} - \frac{1}{\alpha m}}\prt{\bar{z}_{k_2} - \bar{z}_{k_1}} + \frac{1-\cC_0}{\alpha m}e_{k_1:k_2}.
            \end{aligned}
        \end{equation}
    
        Recall that $\ux_{k}= \proxp{\bar{z}_{k}}$ { and $\gamma<\rho^{-1}$}. We have from Lemma \ref{lem:prox_wcvx} that 
        \begin{equation}
            \label{eq:inner_prox}
            \begin{aligned}
                -\inpro{\bar{z}_{k_2} - \bar{z}_{k_1}, \ux_{k_2} - \ux_{k_1}}\leq -(1-\gamma\rho)\norm{\ux_{k_2} - \ux_{k_1}}^2.
            \end{aligned}
        \end{equation}
    
        Combining \eqref{eq:Fnor2zbar_tw} and \eqref{eq:inner_prox}, we obtain
        \small
        \begin{equation}
            \label{eq:inner_ux}
            \begin{aligned}
                &\inpro{\tc, \ux_{k_2} - \ux_{k_1}}
                \leq 
                \frac{1-\cC_0}{\alpha m}\norm{e_{k_1:k_2}}^2\\
                & + \brk{\frac{1-\cC_0}{4\alpha m}- \prt{1 - \cC_0}\prt{\frac{1}{\alpha m} - \frac{1}{\gamma}}(1-\gamma\rho)}\norm{\ux_{k_2} - \ux_{k_1}}^2,
            \end{aligned}
        \end{equation}\normalsize
        where we invoked Young's inequality.
        
        For the remaining inner products in \eqref{eq:cH_k1k2_can0}, we have from Young's inequality that 
        \small
        \begin{equation}
            \label{eq:nf_e_inner}
            \begin{aligned}
                &\frac{\cC_0}{\gamma}\inpro{\gamma\brk{\nabla f\prti{\ux_{k_2}} - \nabla f\prti{\ux_{k_1}}},e_{k_1:k_2}}\leq \frac{\cC_0 L}{2}\norm{e_{k_1:k_2}}^2\\
                & + \frac{\cC_0}{2L}\norm{\nabla f\prti{\ux_{k_2}} - \nabla f\prti{\ux_{k_1}}}^2\\
                &\leq \frac{\cC_0 L}{2}\prt{\norm{\ux_{k_1} - \ux_{k_2}}^2 + \norm{e_{k_1:k_2}}^2},\text{ and }
            \end{aligned}
        \end{equation}
        \begin{equation}
            \label{eq:Fnor_nfe_inner}
            \begin{aligned}
                &\inpro{\Fnor{\bar{z}_{k_1}}, \gamma\brk{\nabla f\prti{\ux_{k_2}} - \nabla f\prti{\ux_{k_1}}} + e_{k_1:k_2}}\\
                &\leq \frac{\alpha m}{2}\norm{\Fnor{\bar{z}_{k_1}}}^2 + \frac{\gamma^2 L^2}{\alpha m}\norm{\ux_{k_1} - \ux_{k_2}}^2 + \frac{1}{\alpha m}\norm{e_{k_1:k_2}}^2.
            \end{aligned}
        \end{equation}\normalsize

        Note that $\cC_0= \frac{3-4\gamma\rho}{2\prt{3 - 4\gamma\rho+4\gamma^2 L^2}}$ and $\gamma\leq 1/[5(\rho + L)]$. We have $\cC_0< 1/2$.
        Substituting \eqref{eq:inner_ux}-\eqref{eq:Fnor_nfe_inner} into \eqref{eq:cH_k1k2_can0} and rearranging terms yield the desired result.
    
    \subsection{Proof of Lemma \ref{lem:ek1k2}}
    \label{app:ek1k2}
    
    The proof is divided into five steps. The first step outlines the roadmap, identifying three key components necessary to establish the desired result (see \eqref{eq:ek1k2_s2}). The subsequent three steps address these components individually. Specifically, \textbf{Step II} presents bounds for two variance terms (see \eqref{eq:sk1p_ub} and \eqref{eq:bsk1p_split}); \textbf{Step III} and \textbf{Step IV} derive bounds for the summation terms (see \eqref{eq:Fnor_L_zbar_sum} and \eqref{eq:cons_tw_norm_sum}, respectively). Finally, \textbf{Step~V} combines these components to obtain the desired result.
    
    \textbf{Step I: Exploring components that bound $\condEi{\normi{e_{k_1:k_2}}^2}{\cF_{k_1}}$.} We first explore the ingredients that bound the term $\condEi{\normi{e_{k_1:k_2}}^2}{\cF_{k_1}}$. 
    { Denote $s_{k_1:k_2}^2:= \normi{\sum_{q=k_1}^{k_2-1} \alpha \bar{\delta}_{q}}^2$ with $\bar{\delta}_p := \sumn\delta_{i,p}/n$.} Noting that $k_2 - k_1 = m$, we have from \eqref{eq:ek1k2_def} and { Lemma~\ref{lem:Fnor_Lip} that
    \small
    \begin{equation}
        \label{eq:ek1k2_s2}
        \begin{aligned}
            \frac{\normi{e_{k_1:k_2}}^2}{3} &\leq \alpha^2 m\Ln^2\sum_{p=k_1}^{k_2-1}\brk{\normi{\bar{z}_{k_1}-\bar{z}_p}^2 + \frac{\normi{\Pi\z_p}^2}{n}} + s_{k_1:k_2}^2.
        \end{aligned}
    \end{equation}\normalsize
    }
    
    \textbf{Step II: Bounding two variance terms.} In this step, we bound the two variance terms $s_{k_1:p}^2$ and $\s_{k_1:p}^2:=\normi{\alpha\sum_{q=k_1}^{p - 1}\Gamma^{p - 1-q} { D_1}\Delta_q}^2$ for any integer $k_1< p\leq k_2$. To deal with the consensus error $\normi{\Pi\z_p}^2$ that appears in the second term in \eqref{eq:ek1k2_s2}, we bound $\normi{\bu_p}^2$ due to $\normi{\Pi\z_p}^2\leq \normi{V}^2\normi{\bu_p}^2$, $\forall k_1 < p\leq k_2$. Such a step requires bounding $\condEi{\s_{k_1:p}^2}{\cF_{k_1}}$ according to \eqref{eq:cons_tw}. 
    
        
    
    { Applying the tower property }and noting that the independent stochastic gradients across agents (Assumption \ref{as:abc}) yields 
    \small
    \begin{equation}
        \label{eq:sk1p_s1}
        \begin{aligned}
            &\condE{s_{k_1:p}^2}{\cF_{k_1}} = \frac{\alpha^2}{n^2}\sum_{q = k_1}^{p-1}\sumn\condE{\condE{\norm{\delta_{i,q}}^2}{\cF_{q}}}{\cF_{k_1}}\\
            &\leq \frac{\alpha^2}{n^2}\sum_{q = k_1}^{p-1}\sumn \condE{\C\brk{f_i(x_{i,q}) - f_i^*} + \sigma^2 }{\cF_{k_1}}.
        \end{aligned}
    \end{equation}\normalsize
    
    According to Assumption \ref{as:smooth} that each $f_i$ is $L$-smooth, we have from the descent lemma and $\normi{\nabla f_i(x)}^2\leq 2L\prti{f_i(x) - f_i^*}$ that $f_i(x_{i,q}) - f_i^* \leq 2\brki{f_i(\bar{x}_q) - f_i^*} + L\normi{x_{i, q} - \bar{x}_q}^2$, for any $q>0$.
    Taking the average on both sides of it among $i=1,2,\ldots,n$, invoking Lemma \ref{lem:fbarx2fproxz} and \eqref{eq:consx2consz_org} lead to 
    \small
    \begin{equation}
        \label{eq:avgfxip2fproxz}
        \begin{aligned}
            \frac{1}{n}\sumn\brk{f_i(x_{i,q}) - f_i^* }
            &\leq 4\brk{\psi\prt{\ux_q} - \bpsi} + \frac{4L}{n}\norm{\Pi\z_q}^2 + 2\sigfn.
        \end{aligned}
    \end{equation}\normalsize
    
    
    Recall that $\ux_k = \proxp{\bar{z}_k}$ for any $k\geq 0$. 
    { Similar to the derivations in Step I in Appendix \ref{app:cH_descent},} we have for $k_2\geq q>k_1$ that 
    \small
    \begin{align}
        \psi(\ux_{q}) - \psi(\ux_{k_1}) 
        &\leq \frac{1}{\gamma}\norm{\bar{z}_{k_1} - \bar{z}_q}^2 + \frac{\gamma}{2}\norm{\Fnor{\bar{z}_{k_1}}}^2,\label{eq:psi2k1}
    \end{align}
    \normalsize
    where we invoked Lemma \ref{lem:prox_wcvx} and applied $\gamma\leq (2-\sqrt{2})/(2\rho)$.
    Substituting \eqref{eq:avgfxip2fproxz} and \eqref{eq:psi2k1} into \eqref{eq:sk1p_s1}, we have for any integer $k_1 <p\leq k_2$ that
    \small
    \begin{equation}
        \label{eq:sk1p_ub}
        \begin{aligned}
            &\condE{s_{k_1:p}^2}{\cF_{k_1}}
            \leq \frac{4\alpha^2\C }{n\gamma}\sum_{q=k_1}^{p-1}\condE{\norm{\bar{z}_{k_1} - \bar{z}_q}^2}{\cF_{k_1}} \\
            &\frac{4\alpha^2 \C\prt{p-k_1}}{n}\brk{\psi\prti{\ux_{k_1}} - \bpsi} + \frac{4\alpha^2\C L}{n^2}\sum_{q=k_1}^{p-1}\condE{\norm{\Pi\z_q}^2}{\cF_{k_1}}\\
            & + \frac{2\alpha^2\C \gamma (p - k_1)}{n}\norm{\Fnor{\bar{z}_{k_1}}}^2 + \frac{\alpha^2 (p-k_1)\prt{2\C\sigfn + \sigma^2}}{n}.
        \end{aligned}
    \end{equation}\normalsize
    
    We next bound the term $\condEi{\s_{k_1:p}^2}{\cF_{k_1}}$. Similar to the derivations for {$\condEi{s_{k_1:p}^2}{\cF_{k_1}}$ in \eqref{eq:sk1p_s1}-\eqref{eq:sk1p_ub}}, we have 
    \small
    \begin{equation}
        \label{eq:bsk1p_split}
        \begin{aligned}
            &\condE{\s_{k_1:p}^2}{\cF_{k_1}} 
            \leq { \alpha^2\norm{D_1}^2\sum_{q=k_1}^{p - 1}\beta^{2(p - 1 - q)}\sumn\condE{\norm{\delta_{i,q}}^2}{\cF_{k_1}}}\\
            &\leq 4\alpha^2 \C L {\norm{D_1}^2} \sum_{q=k_1}^{p-1}\beta^{2(p-1-q)}\condE{\norm{\Pi \z_q}^2}{\cF_{k_1}}\\
            & + 4\alpha^2 n\C {\norm{D_1}^2}\beta_{0:(p-k_1)}^2\brk{\psi\prt{\ux_{k_1}}-\bpsi + \frac{\gamma}{2}\norm{\Fnor{\bar{z}_{k_1}}}^2}\\
            & + 4\alpha^2n\C\gamma^{-1} {\norm{D_1}^2}\sum_{q=k_1}^{p-1}\beta^{2(p-1-q)}\condE{\norm{\bar{z}_{k_1} - \bar{z}_q}^2}{\cF_{k_1}}\\
            & + \alpha^2n{\norm{D_1}^2}\beta_{0:(p-k_1)}^2\prt{2\C\sigfn + \sigma^2}.
        \end{aligned}
    \end{equation}\normalsize

    \textbf{Step III: Bounding  $\sum_{p=k_1}^{k_2 - 1}\condEi{\normi{\bar{z}_{k_1} - \bar{z}_p}^2}{\cF_{k_1}}$}. 
    
    Note that \eqref{eq:ek1k2_s2} also holds for any integer $k_1 <p\leq k_2$. Then, according to $\bar{z}_{p} = \bar{z}_{k_1} - \alpha (p - k_1)\Fnor{\bar{z}_{k_1}} + e_{k_1:p}$, we have 
    \small
    {
    \begin{align}
        &\norm{\bar{z}_{k_1} - \bar{z}_p}^2 
        \leq 2(p - k_1)^2\alpha^2 \norm{\Fnor{\bar{z}_{k_1}}}^2 + 6s_{k_1:p}^2 \nonumber\\
        &+ 6\alpha^2(p-k_1)\Ln^2\sum_{q=k_1}^{p-1}\brk{\norm{{\bar{z}_{k_1}} - {\bar{z}_q}}^2 + \frac{\norm{\Pi \z_q}^2}{n}}.\label{eq:Fnor_L_zbar}
    \end{align}\normalsize
    }
    
    Summing on both sides of \eqref{eq:Fnor_L_zbar} from $p=k_1$ to $k_2 - 1$ leads to 
    {
    \small
    \begin{equation}
        \label{eq:Fnor_L_zbar_sum}
        \begin{aligned}
            &\sum_{p=k_1}^{k_2 - 1}\norm{{\bar{z}_{k_1}} -{\bar{z}_p}}^2\leq \alpha^2 m^3\norm{\Fnor{\bar{z}_{k_1}}}^2 +  6\sum_{p=k_1}^{k_2 -1}s_{k_1:p}^2  \\
            &+ 3\alpha^2m^2\Ln^2\sum_{p=k_1}^{k_2 - 1} \brk{\norm{{\bar{z}_{k_1}} - {\bar{z}_p}}^2 + \frac{\norm{\Pi \z_p}^2}{n}}.
        \end{aligned}
    \end{equation}\normalsize
    }
    
    \textbf{Step IV: Bounding $\sum_{p=k_1}^{k_2 - 1}\condEi{\normi{\bu_p}^2}{\cF_{k_1}}/n$.}
    
    Noting $\bar{z}_{q + 1} - \bar{z}_q = -\alpha \Fnor{\bar{z}_{k_1}} + \alpha\brki{\Fnor{\bar{z}_{k_1}} - \Fnor{\bar{z}_q}} - \alpha \sumn\brki{\Fnori{i}{z_{i,q}} - {\Fnori{i}{\bar{z}_q}}}/n - \alpha \bar{\delta}_q$ { and invoking Lemma~\ref{lem:Fnor_Lip} leads to 
    \begin{equation}
        \label{eq:zq_diff}
        \begin{aligned}
            \norm{\bar{z}_{q + 1} - \bar{z}_q}^2 &\leq 4\alpha^2\norm{\Fnor{\bar{z}_{k_1}}}^2 + 4\alpha^2\Ln^2\norm{\bar{z}_{k_1} - \bar{z}_q}^2\\
            &\quad + \frac{4\alpha^2 \Ln^2}{n}\norm{\Pi \z_q}^2 + 4\alpha^2\norm{\bar{\delta}_q}^2.
        \end{aligned}
    \end{equation}
    }

    Then, for any $k_1<p\leq k_2$, we have from \eqref{eq:cons_tw} { and Lemma \ref{lem:Fnor_Lip}} that  
    \small
    \begin{equation}
        \label{eq:cons_tw_norm}
        \begin{aligned}
            &\condE{\norm{\bu_{p}}^2}{\cF_{k_1}} 
            \leq  4\beta^{2(p - k_1)}\norm{\bu_{k_1}}^2 + 4\condE{\s_{k_1:p}^2}{\cF_{k_1}}\\
            &+ 4\alpha^2 \prt{p - k_1}\Ln^2{\norm{D_1}^2}\sum_{q=k_1}^{p-1}\beta^{2(p-1-q)} \condE{\norm{\Pi\z_q}^2}{\cF_{k_1}}\\
            & + { 4\alpha^2 n\prt{p-k_1} \norm{D_2}^2 \Ln^2 \sum_{q=k_1}^{p-1}\beta^{2(p-1-q)} \norm{\bar{z}_{q + 1} - \bar{z}_q}^2}.
        \end{aligned}
    \end{equation}\normalsize
    
    Note that $\beta<1$. Then, we have $\sum_{p=k_1}^{k_2 - 1}\beta_{0:(p-k_1)}^2\leq m/(1-\beta)$ and $\sum_{p=k_1}^{k_2 - 1}\prt{p-k_1}\sum_{q=k_1}^{p-1}\beta^{2(p-1-q)}a_q\leq \sum_{p=k_1}^{k_2 - 1}\prt{p-k_1}\sum_{q=k_1}^{k_2-1}a_q \leq \frac{m^2}{2}\sum_{q=k_1}^{k_2-1}a_q$, for any $a_q\geq 0$.
    
    Summing on both sides of \eqref{eq:cons_tw_norm} from $p=k_1 + 1$ to $k_2-1$ and invoking the above summations leads to 
    \small
    \begin{equation}
        \label{eq:cons_tw_norm_sum}
        \begin{aligned}
            &\sum_{p=k_1}^{k_2 - 1}\condE{\norm{\bu_p}^2}{\cF_{k_1}} 
            = \norm{\bu_{k_1}}^2 + \sum_{p=k_1 + 1}^{k_2 - 1}\condE{\norm{\bu_p}^2}{\cF_{k_1}}\\
            &\leq 4\beta_{0:m}^2\norm{\bu_{k_1}}^2 +  4\sum_{p=k_1 + 1}^{k_2 - 1}\condE{\s_{k_1:p}^2}{\cF_{k_1}} \\
            & + 8\alpha^4m^2n {\norm{D_2}^2}\Ln^2\sum_{q=k_1}^{k_2-1}\prt{\norm{\Fnor{\bar{z}_{k_1}}}^2  \right.\\
            &\left.+ \Ln^2\condE{\norm{\bar{z}_q - \bar{z}_{k_1}}^2}{\cF_{k_1}}} \\
            &{ + 16\alpha^4 n\norm{D_2}^2\Ln^2 \sum_{p=k_1}^{k_2 - 1}(p-k_1)\sum_{q=k_1}^{p-1}\beta^{2(p-1-q)}\condE{\norm{\bar{\delta}_q}^2}{\cF_{k_1}}} \\
            &+ 2\alpha^2m^2\Ln^2\prt{\norm{D_1}^2 + 4\alpha^2\Ln^2\norm{D_2}^2 }\sum_{q=k_1}^{k_2 -1}\condE{\norm{\Pi\z_q}^2}{\cF_{k_1}}.
        \end{aligned}
    \end{equation}\normalsize
    
    
    
    \textbf{Step V: Bounding $\condEi{\normi{e_{k_1:k_2}}^2}{\cF_{k_1}}$.} In this step, we combine the ingredients derived in \textbf{Step II-Step IV} to obtain the desired result.
    
    Note that $\normi{\Pi\z_p}^2\leq \normi{V}^2\normi{\bu_p}^2$ for any $p>0$. Combining \eqref{eq:cons_tw_norm_sum} and \eqref{eq:Fnor_L_zbar_sum}, { and letting $\alpha\leq 1/(4\normi{D_2}\Ln)$} leads to 
    \small
    \begin{equation}
        \label{eq:sumk1k2}
        \begin{aligned}
            &\condE{\cM_{k_1:k_2}}{\cF_{k_1}}\leq \frac{4\beta_{0:m}^2}{n}\norm{\bu_{k_1}}^2  + \frac{4}{n}\sum_{p=k_1}^{k_2 - 1}\condE{\s_{k_1:p}^2}{\cF_{k_1}} \\
            &+6\sum_{p=k_1}^{k_2 - 1}\condE{s_{k_1:p}^2}{\cF_{k_1}} +  \frac{3}{2}\alpha^2 m^3\norm{\Fnor{\bar{z}_{k_1}}}^2\\ 
            & + 4\alpha^2m^2\norm{V}^{2}\Ln^2\prt{\norm{D_1}^2 + 1}\sum_{p=k_1}^{k_2 - 1}\frac{1}{n}\condE{\norm{\bu_p}^2}{\cF_{k_1}}  \\
            &+ 4\alpha^2m^2\Ln^2 \sum_{p=k_1}^{k_2 - 1}\condE{\norm{\bar{z}_{k_1} - \bar{z}_p}^2}{\cF_{k_1}}\\
            & { + 16\alpha^4 \norm{D_2}^2\Ln^2 \sum_{p=k_1}^{k_2 - 1}(p-k_1)\sum_{q=k_1}^{p-1}\beta^{2(p-1-q)}\condE{\norm{\bar{\delta}_q}^2}{\cF_{k_1}}}.
        \end{aligned}
    \end{equation}\normalsize
    
    { Similar to the derivations in \eqref{eq:sk1p_s1}-\eqref{eq:sk1p_ub}, we have 
    \small
    \begin{equation}
        \label{eq:wg2psi}
        \begin{aligned}
            &\sum_{q=k_1}^{p - 1}\beta^{2(p- 1-q)}\condE{\norm{\bar{\delta}_q}^2}{\cF_{k_1}} \leq \frac{4\C \beta_{0:p}^2}{n}\brk{\psi\prt{\ux_{k_1}} - \bpsi}  \\
            &+ \frac{2\C \gamma \beta_{0:p}^2}{n}\norm{\Fnor{\bar{z}_{k_1}}}^2 + \frac{4\C}{n\gamma} \sum_{q=k_1}^{p-1}\condE{\norm{\bar{z}_{k_1} - \bar{z}_q}^2}{\cF_{k_1}} \\
            &+ \frac{4\C L}{n^2} \sum_{q=k_1}^{p-1}\condE{\norm{\Pi\z_q}^2}{\cF_{k_1}} + \frac{\prt{2\C\sigfn + \sigma^2} \beta_{0:p}^2}{n}.
        \end{aligned}
    \end{equation}\normalsize
    }
    
    Substituting \eqref{eq:sk1p_ub}, \eqref{eq:bsk1p_split}, and \eqref{eq:wg2psi} into \eqref{eq:sumk1k2}, { noting $\beta_{0:p}^2\leq \max\crki{1/(1-\beta^2), p}$, and letting $\alpha\leq 1/(16m\Ln\sqrt{\normi{D_1}^2 + 1})$, $\gamma\leq 1/[5(\C + L + \rho)]$} yields 
    \small
    \begin{align}
        &\tc_1 \sum_{p=k_1}^{k_2 - 1}\condE{\norm{\bar{z}_{k_1} - \bar{z}_p}^2}{\cF_{k_1}} + \tc_2\sum_{p=k_1}^{k_2 - 1}\frac{1}{n}\condE{\norm{\bu_p}^2}{\cF_{k_1}} \nonumber\\
        &\leq \frac{4\beta_{0:m}^2}{n}\norm{\bu_{k_1}}^2 + \alpha^2m^3\prt{ 2 + \norm{D_1}^2}\norm{\Fnor{\bar{z}_{k_1}}}^2\nonumber\\
        & + 4\alpha^2m \brk{\norm{D_1}^2\beta_{0:m}^2 + \frac{3m}{4n} + \frac{2\alpha^2m\normi{D_2}^2 \Ln^2}{n(1-\beta^2)}} \nonumber\\
        &\cdot\crk{4\C \brk{\psi\prt{\ux_{k_1}} - \bpsi} + 2\C\sigfn + \sigma^2},\label{eq:sumk1k2_AA}
    \end{align}\normalsize
    { where 
    $\tc_1 := 1 - 16\alpha^2mC_0\gamma^{-1}\normi{D_1}^2 - 24\alpha^2 mC_0n^{-1}\gamma^{-1}- 4\alpha^2m^2\Ln^2 - 32\alpha^4m^2\normi{D_2}^2\Ln^2C_0 n^{-1}\gamma^{-1}$ and $\tc_2:= 1 - 16\alpha^2 m C_0 L\normi{D_1}^2\normi{V}^2 - 24\alpha^2 m C_0 L \normi{V}^2n^{-1} - 32\alpha^4m^2\Ln^2 C_0 L \normi{D_2}^2\normi{V}^2n^{-1} - 4\alpha^2m^2\normi{V}^2\Ln^2\prti{\normi{D_1}^2 + 1}.$
    
    }

    It suffices to have {$\gamma\leq 1/[5(\C + L + \rho)]$ and $\alpha\leq \min\crki{1/(16m\Ln\sqrt{\normi{D_1}^2 + 1}), \sqrt{1-\beta^2}/(4\normi{D_2}\Ln)}$} such that $\tc_1\geq 1/2$, $\tc_2\geq 1/2$, { and $\alpha^2 m\normi{D_2}^2\Ln^2\leq m(1-\beta^2)/16$, which} yields the desired result \eqref{eq:sumk1k2_ub}.
    
    Note that $\gamma\leq 1/(2L)$. 
    Taking conditional expectation on both sides of \eqref{eq:ek1k2_s2} and substituting \eqref{eq:sk1p_ub} into it leads to 
    \small
    \begin{equation}
        \label{eq:Eek1k2}
        \begin{aligned}
            &\condE{\norm{e_{k_1:k_2}}^2}{\cF_{k_1}}\leq 3\alpha^2 m \ccD_1^2{\condE{\cM_{k_1:k_2}}{\cF_{k_1}}} \\
            & + \frac{12\alpha^2 m \C}{n}\brk{\psi\prt{\ux_{k_1}} - \bpsi}   + \frac{6\alpha^2 m \C\gamma}{n} \norm{\Fnor{\bar{z}_{k_1}}}^2\\
            &  + \frac{3\alpha^2 m\prt{2\C\sigfn + \sigma^2}}{n}.
        \end{aligned}
    \end{equation}\normalsize
    
    Substituting \eqref{eq:sumk1k2_ub} into \eqref{eq:Eek1k2} leads to the desired result \eqref{eq:Eek1k2_ub}.
    

    \subsection{Proof of Lemma \ref{lem:cL}}
    \label{app:cL}
    
    Note that $\brki{\psi\prt{\ux_{k_1}} - \bpsi}\leq \cH_{k_1}$ for any $k_1\geq 0$. Taking the conditional expectation on \eqref{eq:cH_descent} and substituting \eqref{eq:Eek1k2_ub} yields
    \small
    \begin{equation}
        \label{eq:cH_ek}
        \begin{aligned}
            &\condE{\cH_{k_2}}{\cF_{k_1}}\leq \prt{ 1 + 12\alpha \C H_1}\cH_{k_1} + { 3\alpha H_1}\prt{2\C\sigfn + \sigma^2} \\
            & - { \alpha m\brk{\frac{\cC_0}{2} - \frac{6\C\gamma}{mn} - 6\alpha^2m^2\ccD_1^2\prt{\norm{D_1}^2 + 2}}}\norm{\Fnor{\bar{z}_{k_1}}}^2\\
            & + \frac{24\alpha {\ccD_1^2}\beta_{0:m}^2\norm{\bu_{k_1}}^2}{n}.
        \end{aligned}
    \end{equation}\normalsize
    
    Substituting \eqref{eq:sumk1k2_ub} into \eqref{eq:Euk_re} and noting $1-\beta^{2m} = \prti{1-\beta^2}\beta_{0:m}^2$ leads to 
    
    \small
    \begin{equation}
        \label{eq:Euk_ub}
        \begin{aligned}
            &\condE{\norm{\bu_{k_2}}^2}{\cF_{k_1}}\leq \prt{\frac{1 + \beta^{2m}}{2} + {\frac{64\alpha^2m \ccD_2^2 }{1-\beta^{2}}}}\norm{\bu_{k_1}}^2 \\
            &+  4\alpha^2n\beta_{0:m}^2 \brk{ { H_2} + \frac{4\alpha^2m^4 {\ccD_2^2(\normi{D_1}^2 + 2)}}{\prt{1-\beta^{2m}}\beta_{0:m}^2}} \norm{\Fnor{\bar{z}_{k_1}}}^2\\
            &+ { 2\alpha^2n\beta_{0:m}^2} \brk{{ H_2 + \frac{32\alpha^2m^2\ccD_2^2 \normi{D_1}^2}{1-\beta^{2m}} + \frac{32\alpha^2 m^3\ccD_2^2}{n\beta^2_{0:m}(1-\beta^{2m})}}}\\
            &\cdot \crk{4\C \brk{\psi\prt{\proxp{\bar{z}_{k_1}}} - \bpsi} + 2\C \sigfn + \sigma^2}.
        \end{aligned}
    \end{equation}\normalsize

    { 
    Substituting \eqref{eq:cH_ek} and \eqref{eq:Euk_ub} into \eqref{eq:cLk1k2} (setting $k=k_2$) leads to
    \begin{equation}
        \label{eq:cL_tcs}
        \begin{aligned}
            &\condE{\cL_{k_2}}{\cF_{k_1}} \leq \prt{1 + 4\tc_5}\cL_{k_1}  + \tc_3\norm{\bu_{k_1}}^2\\
            &-\alpha m\tc_4\norm{\Fnor{\bar{z}_{k_1}}}^2+\tc_5 \prt{2\C\sigfn + \sigma^2},
        \end{aligned}
    \end{equation}
    where 
    \small
    \begin{align*}
        &\tc_3:=\frac{24\alpha \ccD_1^2\beta_{0:m}^2}{n} + \prt{\frac{1 + \beta^{2m}}{2} + \frac{64\alpha^2m \ccD_2^2 }{1-\beta^{2}}} \frac{50\alpha\ccD_1^2}{n\prti{1-\beta^2}},\\
        &\tc_4:= \frac{\cC_0}{2} - \frac{6\C\gamma}{mn} - \frac{800\alpha^4 m^3\ccD_1^2\ccD_2^2\prti{\normi{D_1}^2 + 2}}{(1-\beta^2)(1-\beta^{2m})}\\
        &- 6\alpha^2m^2\ccD_1^2\prt{\norm{D_1}^2 + 2}  - \frac{200\alpha^2 \ccD_1^2 \normi{D_1}^2 \beta^2_{0:m}}{m(1-\beta^2)}\\
        &- \frac{800 \alpha^4\Ln^2\ccD_1^2\normi{D_2}^2}{(1-\beta^2)^2},\qquad \tc_5:=3\alpha H_1 + \frac{100\alpha^3\beta_{0:m}^2\ccD_1^2}{1-\beta^2}\\
        &\cdot\brk{H_2 + \frac{32\alpha^2m^2\ccD_2^2 \normi{D_1}^2}{1-\beta^{2m}} + \frac{96\alpha^2 m^3\ccD_2^2}{n\beta^2_{0:m}(1-\beta^{2m})}}.
    \end{align*}\normalsize

    We now simplify coefficients $\tc_3$--$\tc_5$. Invoking the conditions on $\alpha$ and $\gamma$, we obtain $\tc_3\leq 50\alpha \ccD_1^2/[n(1-\beta^2)]$, $-\alpha m\tc_4\leq \alpha m /9$, and $\tc_5  \leq 5\alpha /n + \cA(\alpha)$. This finishes the proof. 
    }
    

    \subsection{Proof of Lemma \ref{lem:cL_bounded}}
    \label{app:cL_bounded}

    Based on the definition of $\crki{\tk_j}$ in \eqref{eq:tkj}, we have $\tk_{j+1}-\tk_j = m$ for any $j\geq 0$. Therefore, we can apply the recursion \eqref{eq:cL} to $\crki{\tk_j}$. Taking the full expectation on \eqref{eq:cL} yields that for any $(T + 1)\geq j>0$, we have
    {
    \begin{align*}
        &\E\brk{\cL_{\tk_{j}}} 
        \leq \exp\crk{\brk{\frac{20\alpha \C}{n} + 4\cA(\alpha)\C}(T + 1)}\crk{\cL_0  \right.\\
        &\left. + \brk{\frac{5\alpha}{n} + \cA(\alpha)}(T + 1)\prt{2\C\sigfn + \sigma^2}},
    \end{align*}
    }
    where we invoked the relations $\prti{1 + x}^t \leq \exp\prti{x t} \leq \exp\brki{x(T + 1)}$ for any $0\leq t \leq (T + 1)$ and $x\geq 0$.
    
    Note that $T + 1 = (K-Q)/m + 1\leq K/m + 1\leq 2K/m$. Invoking the conditions for $\alpha$, $\gamma$, and $m$ leads to the desired result.
    
    \subsection{Proof of Theorem \ref{thm:abc}}
    \label{app:abc}
    
    The proof is divided into three steps based on \eqref{eq:sum_dist}.
    \textbf{Step~I} applies the uniform bound on $\E\brki{\cL_{\tk_t}}$ to simplify the inner summation in the second term of \eqref{eq:sum_dist}, as shown in \eqref{eq:Fnor_inter_sum}. Then, the remaining task is to bound the sums $\sum_{t=0}^T \E\brki{\normi{\bu_{\tk_t}}^2}$ and $\sum_{t=0}^T \E\brki{\normi{\Fnor{\bar{z}_{\tk_t}}}^2}$. These two terms are addressed in \textbf{Step II} and \textbf{Step III}, as in \eqref{eq:Euk_sum_ub} and \eqref{eq:cL_sum_s1}, respectively.
    
    \textbf{Step I: Bounding $\sum_{t=0}^T\E\brki{\cM_{\tk_t:\tk_{t + 1}}}/K$.}
    
    Invoking \eqref{eq:sumk1k2_ub}, \eqref{eq:cLkj_ub}, and taking the full expectation yields
    {
    \small
    \begin{equation}
        \label{eq:Fnor_inter_sum}
        \begin{aligned}
            &\frac{4\Ln^2\prt{\norm{V}^2 + 1}}{K}\sum_{t=0}^T\E\brk{\cM_{\tk_t:\tk_{t+1}}} \leq \frac{m}{2K}\sum_{t=0}^T\E\brk{\norm{\Fnor{\bar{z}_{\tk_t}}}^2}\\
            & + 128\alpha^2 \Ln^2\prt{\norm{V}^2 + 1}\prt{\norm{D_1}^2\beta_{0:m}^2 + \frac{m}{n}}\tilde{\sigma}^2\\
            &+ \frac{{ 32}\beta_{0:m}^2\Ln^2\prt{\norm{V}^2 + 1}}{nK}\sum_{t=0}^T\E\brk{\norm{\bu_{\tk_t}}^2},
        \end{aligned}
    \end{equation}\normalsize
    }
    where we let {$\alpha\leq 1/(80m\ccD_2)$}.

    \textbf{Step II: Bounding the term $\sum_{t=0}^T \E\brki{\normi{\bu_{\tk_t}}^2}$.}
    In light of the definition of $\hat{\cL}$ in \eqref{eq:cLkj_ub}, the choice of $\alpha$, and $(T + 1)/K\leq 2/m$ for any $T\geq 1$, we sum on both sides of \eqref{eq:Euk_ub} from $t = 0$ to $t= T$ to obtain
    {
    \small
    \begin{equation}
        \label{eq:Euk_sum_ub}
        \begin{aligned}
            &\frac{1}{nK}\sum_{t=0}^T\E\brk{\norm{\bu_{\tk_t}}^2} \leq \frac{24\alpha^2}{m(1-\beta^2)} \brk{H_2 + \frac{32\alpha^2m^2\ccD_2^2 \norm{D_1}^2}{1-\beta^{2m}} \right.\\
            &\left. + \frac{32\alpha^2 m^3\ccD_2^2}{n\beta^2_{0:m}(1-\beta^{2m})}}\tilde{\sigma}^2+  
            \frac{12\alpha^2}{(1-\beta^{2})}\brk{\frac{4\alpha^2m^4 \ccD_2^2(\normi{D_1}^2 + 2)}{\prti{1-\beta^{2m}}\beta_{0:m}^2}\right.\\
            &\left. + H_2} \frac{1}{K}\sum_{t=0}^{T}\E\brk{\norm{\Fnor{\bar{z}_{\tk_t}}}^2}.
        \end{aligned}
    \end{equation}\normalsize
    }
    
    \textbf{Step III: Bounding the outer summation in \eqref{eq:sum_norm}.}
    According to \eqref{eq:cL} and \eqref{eq:cLkj_ub}, we have for any $T\geq t\geq 0$ that 
    {
    \small
    \begin{equation}
        \label{eq:cL_tk}
        \begin{aligned}
            \E\brki{\cL_{\tk_{t + 1}}}&\leq \E\brki{\cL_{\tk_t}}  - \frac{\alpha m}{9}\E\brki{\normi{\Fnor{\bar{z}_{\tk_t}}}^2} + \brk{\frac{5\alpha }{n} + \cA(\alpha)}\tilde{\sigma}^2.
        \end{aligned}
    \end{equation}\normalsize
    }
    Summing on both sides of \eqref{eq:cL_tk} from $t = 0$ to $t = T$ and noting that $(T + 1)/K\leq 2/m$ for any $T\geq 1$, we obtain
    {
    \small
    \begin{equation}
        \label{eq:cL_sum_s1}
        \begin{aligned}
            &\frac{m}{K}\sum_{t=0}^{T}\E\brk{\norm{\Fnor{\bar{z}_{\tk_t}}}^2} 
            \leq \frac{9\cL_{0}}{\alpha K}+ \brk{\frac{90 }{nm} + \frac{18\cA(\alpha)}{\alpha m}}\tilde{\sigma}^2.
        \end{aligned}   
    \end{equation}\normalsize
    }
    
    Substituting \eqref{eq:Fnor_inter_sum}, \eqref{eq:Euk_sum_ub}, \eqref{eq:cL_sum_s1} into \eqref{eq:sum_dist} leads to
    {
    \small
    \begin{equation}
        \label{eq:Fnat_avg}
        \begin{aligned}
            &\frac{1}{nK}\sum_{k=0}^{K-1}\sumn \E\brk{\norm{\frac{\Fnat{x_{i,k}}}{\gamma}}^2}
            \leq \brk{\frac{(7+192/80^2) \normi{D_1}^2}{1-\beta^2}  \right.\\
            &\left. + \frac{m}{n} +\frac{24\alpha^2 \normi{D_2}^2\Ln^2}{n(1-\beta^2)^2} +  \frac{192\alpha^2 m^2\ccD_2^2}{n(1-\beta^2)(1-\beta^{2m})}}\\
            &\cdot \frac{128\alpha^2\Ln^2 (\normi{V}^2 + 1)\tilde{\sigma}^2}{(1-\gamma\rho)^2} + \frac{5m}{(1-\gamma\rho)^2 K}\sum_{t=0}^T\E\brk{\norm{\Fnor{\bar{z}_{\tk_t}}}^2}\\
            &\leq \frac{71\cL_0}{\alpha K} + \frac{704\tilde{\sigma}^2}{mn} + \frac{141\cA(\alpha)\tilde{\sigma}}{\alpha m} + \brk{\frac{(7+192/80^2) \normi{D_1}^2}{1-\beta^2} \right.\\
            &\left.+ \frac{(1 + 192/80^2)m}{n} +\frac{24\alpha^2 \normi{D_2}^2\Ln^2}{n(1-\beta^2)^2}}\frac{128\alpha^2\Ln^2 (\normi{V}^2 + 1)\tilde{\sigma}^2}{(1-\gamma\rho)^2},
        \end{aligned}
    \end{equation}\normalsize
    }
    where we invoked the choice of $\alpha$. 
    Note that {$\normi{\bu_0}^2 \leq\normi{\Fnorb{\1\bar{z}_0^{\T}}}^2/(80^2\ccD_1^2) + 2\normi{V^{-1}}^2\normi{\Pi\z_0}^2$ based on the definition of $\bu_0$ in Lemma \ref{lem:re} given that $\alpha\leq 1/(80\normi{D_2}\ccD_1)$.
    We obtain the upper bound for $\cL_0$ 
    from the definition of $\cL_0$ in \eqref{eq:cLk1k2}. Substituting the above upper bound of $\cL_0$ into \eqref{eq:Fnat_avg}, invoking the choice of $\alpha$ and $\gamma$ leads to the desired result \eqref{eq:E_Prox-grad}.}
    
    Noting the choice of $\alpha$, $\gamma$, and $m$ in \eqref{eq:alpha_order}, we have $1/(1-\beta^{2m})\leq 2$, and 
    \small
    \begin{equation}
        \label{eq:alpha_inv}
        \begin{aligned}
            &\frac{\Delta_\psi}{\alpha K} = {96\sqrt{\frac{\ccD_2\Delta_\psi\tilde{\sigma}^2}{nK}}} + \frac{\eta\Delta_\psi}{K},\; \alpha^2\leq {\frac{n\Delta_\psi}{96^2\ccD_2\tilde{\sigma}^2 K}},\\
            &\frac{\tilde{\sigma}^2}{mn}\leq {\sqrt{\frac{\ccD_2\Delta_\psi\tilde{\sigma}^2}{nK}}}.
        \end{aligned}
    \end{equation} \normalsize
    
    Substituting \eqref{eq:alpha_inv}, {$\ccD_1^2 = \orderi{(\normi{V}^2 + 1)}$ and  $\ccD_2^2 = \orderi{(\normi{V}^2 + 1)(\normi{D_1}^2 + 1)}$} into \eqref{eq:E_Prox-grad} leads to the desired result~\eqref{eq:EFnor_order}.

\end{appendices}

\bibliographystyle{IEEEtran}
\bibliography{references_all}

\onecolumn

\subsection{Proof of Lemma \ref{lem:norm-abc_cons_tw}}
    \label{app:norm-dspgd_cons_tw}

    According to \eqref{eq:cons_tw}, we have 
        \begin{equation}
            \label{eq:cons_tw_norm2}
            \begin{aligned}
                &\norm{\bu_{k_2}}^2 \leq \norm{\Gamma^{k_2 - k_1}\bu_{k_1}}^2 + 2\norm{\alpha\sum_{p=k_1}^{k_2 - 1}\Gamma^{k_2 - 1-p} D_1\Delta_p}^2 \\
                &\quad  + 2\norm{\alpha\sum_{p=k_1}^{k_2 - 1}\Gamma^{k_2 - 1-p}V^{-1}\begin{pmatrix}
                    \hLambda_a \hU^{\T}\brk{\Fnorb{\z_p} - \Fnorb{\barz_p} }\\
                    \hLambda_b^{-1}\hLambda_a\hU^{\T}\brk{\Fnorb{\barz_{p + 1}} - \Fnorb{\barz_p}}
                \end{pmatrix}}^2\\
                &\quad +2\inpro{\Gamma^{k_2 - k_1}\bu_{k_1}, \alpha\sum_{p=k_1}^{k_2 - 1}\Gamma^{k_2 - 1-p}V^{-1}\begin{pmatrix}
                    \hLambda_a \hU^{\T}\brk{\Fnorb{\z_p} - \Fnorb{\barz_p} }\\
                    \hLambda_b^{-1}\hLambda_a\hU^{\T}\brk{\Fnorb{\barz_{p + 1}} - \Fnorb{\barz_p}}
                \end{pmatrix}}\\
                &\quad + 2\inpro{\Gamma^{k_2 - k_1}\bu_{k_1}, \alpha\sum_{p=k_1}^{k_2 - 1}\Gamma^{k_2 - 1-p} D_1\Delta_p}.
            \end{aligned}
        \end{equation}
    
        Note that $\condE{\Delta_p}{\cF_p} = \0$ based on Assumption~\ref{as:abc}. Then, we have from tower property that $\condEi{\inproi{\Gamma^{k_2 - k_1}\bu_{k_1}, \sum_{p=k_1}^{k_2 - 1}\Gamma^{k_2 - 1-p} D_1\Delta_p}}{\cF_{k_1}} = 0$.
        
        For the other inner products in \eqref{eq:cons_tw_norm2}, we have from Young's inequality and Lemma \ref{lem:Fnor_Lip} that
        \begin{align}
            &2\inpro{\Gamma^{k_2 - k_1}\bu_{k_1}, \alpha\sum_{p=k_1}^{k_2 - 1}\Gamma^{k_2 - 1-p}V^{-1}\begin{pmatrix}
                \hLambda_a \hU^{\T}\brk{\Fnorb{\z_p} - \Fnorb{\barz_p} }\nonumber \\
                \hLambda_b^{-1}\hLambda_a\hU^{\T}\brk{\Fnorb{\barz_{p + 1}} - \Fnorb{\barz_p}}
            \end{pmatrix}}\\
                &\leq c\norm{\Gamma^{k_2 - k_1}\bu_{k_1}}^2 + \frac{1}{c}\norm{\alpha\sum_{p=k_1}^{k_2 - 1}\Gamma^{k_2 - 1-p}D_1\brk{\Fnorb{\z_p} - \Fnorb{\barz_p}}}^2\nonumber \\
                &\quad + \frac{1}{c}\norm{\alpha\sum_{p=k_1}^{k_2 - 1}\Gamma^{k_2 - 1-p}D_2\brk{\Fnorb{\barz_{p + 1}} - \Fnorb{\barz_p}}}^2\nonumber \\
                &\leq c \beta^{2m}\norm{\bu_{k_1}}^2 + \frac{\alpha^2 m \Ln^2 \prt{\norm{D_1}^2 + 4\alpha^2\Ln^2 \norm{D_2}^2}}{c}\sum_{p=k_1}^{k_2 - 1} \norm{\Pi\z_p}^2 + \frac{4\alpha^4 m n \norm{D_2}^2\Ln^4}{c}\sum_{p=k_1}^{k_2 -1}\norm{\bar{z}_{k_1} - \bar{z}_p}^2 \nonumber \\
                &\quad + \frac{4\alpha^4 m n \norm{D_2}^2\Ln^2 \beta^2_{0:m}}{c}\norm{\Fnor{\bar{z}_{k_1}}}^2 + \frac{4\alpha^4 m n \norm{D_2}^2\Ln^2}{c}\sum_{p=k_1}^{k_2 - 1}\beta^{2(k_2 - 1 - p)}\norm{\bar{\delta}_p}^2,\;\forall c>0,\label{eq:cons_inner}
        \end{align}
        where we invoked \eqref{eq:zq_diff} in the last inequality.
    
    
        In light of $\normi{\Pi\z_p}^2\leq \normi{V}^2\normi{\bu_p}^2$ for any $p>0$, combining \eqref{eq:cons_tw_norm2}-\eqref{eq:cons_inner}, and letting $c = (1-\beta^{2m})/(2\beta^{2m})$ lead to 
        \begin{equation}
            \label{eq:Euk}
            \begin{aligned}
                &\condE{\norm{\bu_{k_2}}^2}{\cF_{k_1}} 
                \leq \frac{1 + \beta^{2m}}{2}\norm{\bu_{k_1}}^2 + 2\s_{k_1:k_2} + \frac{2\alpha^2 m \Ln^2 \prt{\norm{D_1}^2 + 4\alpha^2\Ln^2 \norm{D_2}^2}}{1-\beta^{2m}}\sum_{p=k_1}^{k_2 - 1} \norm{\Pi\z_p}^2 \\
                &\quad + \frac{8\alpha^4 m n \norm{D_2}^2\Ln^4}{1-\beta^{2m}}\sum_{p=k_1}^{k_2 -1}\norm{\bar{z}_{k_1} - \bar{z}_p}^2+ \frac{8\alpha^4 m n \norm{D_2}^2\Ln^2 \beta^2_{0:m}}{1-\beta^{2m}}\norm{\Fnor{\bar{z}_{k_1}}}^2 \\
                &\quad + \frac{8\alpha^4 m n \norm{D_2}^2\Ln^2}{1-\beta^{2m}}\sum_{p=k_1}^{k_2 - 1}\beta^{2(k_2 - 1 - p)} \norm{\bar{\delta}_p}^2\\
                &\leq \frac{1 + \beta^{2m}}{2}\norm{\bu_{k_1}}^2  + \brk{\frac{2\alpha^2 m \Ln^2 \prt{\norm{D_1}^2 + 4\alpha^2\Ln^2 \norm{D_2}^2}}{1-\beta^{2m}} + 8\alpha^2 \C L\norm{D_1}^2 + \frac{32\alpha^4 m \norm{D_2}^2\Ln^2\C L}{n(1-\beta^{2m})} } \sum_{p=k_1}^{k_2 - 1} \norm{\Pi\z_p}^2\\
                &\quad + \prt{ \frac{8\alpha^4 m n \norm{D_2}^2\Ln^4}{1-\beta^{2m}} + 8\alpha^2 n\C\gamma^{-1}\norm{D_1}^2 + \frac{32\alpha^4 m \norm{D_2}^2\Ln^2\C \gamma^{-1}}{1-\beta^{2m}}}\sum_{p=k_1}^{k_2 -1}\norm{\bar{z}_{k_1} - \bar{z}_p}^2\\
                &\quad + \prt{\frac{8\alpha^4 m n \norm{D_2}^2\Ln^2\beta_{0:m}^2}{1-\beta^{2m}} + 4\alpha^2n\C\norm{D_1}^2\gamma\beta_{0:m}^{2} + \frac{16\alpha^4 m \norm{D_2}^2\Ln^2\C \gamma\beta_{0:m}^2}{1-\beta^{2m}} }\norm{\Fnor{\bar{z}_{k_1}}}^2\\
                &\quad + \prt{8\alpha^2 n\C\norm{D_1}^2\beta_{0:m}^2 +  \frac{32\alpha^4 m \norm{D_2}^2\Ln^2\C\beta_{0:m}^2}{1-\beta^{2m}} }\brk{\psi\prt{\proxp{\bar{z}_{k_1}}} - \bpsi}\\
                &\quad + \prt{2\alpha^2n \norm{D_1}^2\beta_{0:m}^2 + \frac{8\alpha^4 m \norm{D_2}^2\Ln^2\beta_{0:m}^2}{1-\beta^{2m}} }\prt{2\C\sigfn + \sigma^2}\\
                &\leq \frac{1 + \beta^{2m}}{2}\norm{\bu_{k_1}}^2  + \frac{2\alpha^2 m}{1-\beta^{2m}}\brk{\Ln^2 \prt{\norm{D_1}^2 + 1} + 4 \C L\prt{\norm{D_1}^2 + 1}} \sum_{p=k_1}^{k_2 - 1} \norm{\Pi\z_p}^2\\
                &\quad + \frac{8\alpha^2 mn}{1-\beta^{2m}}\prt{\Ln^2 + \frac{\C\normi{D_1}^2}{m\gamma} + \frac{\C}{n\gamma}}\sum_{p=k_1}^{k_2 -1}\norm{\bar{z}_{k_1} - \bar{z}_p}^2\\
                &\quad + 4\alpha^2n\beta_{0:m}^2\prt{\frac{2\alpha^2 m \norm{D_2}^2\Ln^2}{1-\beta^{2m}} + \C\norm{D_1}^2\gamma + \frac{4\alpha^2 m \norm{D_2}^2\Ln^2\C \gamma}{1-\beta^{2m}} }\norm{\Fnor{\bar{z}_{k_1}}}^2\\
                &\quad + 8\alpha^2n\beta_{0:m}^2 \C\prt{\norm{D_1}^2 +  \frac{4\alpha^2 m \norm{D_2}^2\Ln^2}{n(1-\beta^{2m})} }\brk{\psi\prt{\proxp{\bar{z}_{k_1}}} - \bpsi}\\
                &\quad + 2\alpha^2 n\beta_{0:m}^2 \prt{\norm{D_1}^2+ \frac{4\alpha^2 m \norm{D_2}^2\Ln^2}{n(1-\beta^{2m})} }\prt{2\C\sigfn + \sigma^2},
            \end{aligned}
        \end{equation}
        where we let $\alpha\leq 1/(4\normi{D_2}\Ln)$, and invoked \eqref{eq:bsk1p_split} and \eqref{eq:wg2psi} in the last inequality. 
        Letting $\alpha\leq 1/(4m\Ln)$ and $\gamma\leq 1/(4\C)$ yields the desired result.


\subsection{Detailed Proof of Lemma \ref{lem:cL}}
    \label{app:cL_detail}
    
    Note that $\brki{\psi\prt{\ux_{k_1}} - \bpsi}\leq \cH_{k_1}$ for any $k_1\geq 0$. Taking the conditional expectation on \eqref{eq:cH_descent} and substituting \eqref{eq:Eek1k2_ub} yields
    \begin{equation}
        \label{eq:cH_ek_d}
        \begin{aligned}
            &\condE{\cH_{k_2}}{\cF_{k_1}}\leq \prt{ 1 + 12\alpha \C H_1}\cH_{k_1} + { 3\alpha H_1}\prt{2\C\sigfn + \sigma^2} \\
            & - { \alpha m\brk{\frac{\cC_0}{2} - \frac{6\C\gamma}{mn} - 6\alpha^2m^2\ccD_1^2\prt{\norm{D_1}^2 + 2}}}\norm{\Fnor{\bar{z}_{k_1}}}^2+ \frac{24\alpha {\ccD_1^2}\beta_{0:m}^2\norm{\bu_{k_1}}^2}{n} .
        \end{aligned}
    \end{equation}
    
    Substituting \eqref{eq:sumk1k2_ub} into \eqref{eq:Euk_re} and noting $1-\beta^{2m} = \prti{1-\beta^2}\beta_{0:m}^2$ leads to 
    
    \begin{equation}
        \label{eq:Euk_ub_d}
        \begin{aligned}
            &\condE{\norm{\bu_{k_2}}^2}{\cF_{k_1}}\leq \prt{\frac{1 + \beta^{2m}}{2} + {\frac{64\alpha^2m \ccD_2^2 }{1-\beta^{2}}}}\norm{\bu_{k_1}}^2 +  4\alpha^2n\beta_{0:m}^2 \brk{ { H_2} + \frac{4\alpha^2m^4 {\ccD_2^2(\normi{D_1}^2 + 2)}}{\prt{1-\beta^{2m}}\beta_{0:m}^2}} \norm{\Fnor{\bar{z}_{k_1}}}^2\\
            &+ { 2\alpha^2n\beta_{0:m}^2} \brk{{ H_2 + \frac{32\alpha^2m^2\ccD_2^2 \normi{D_1}^2}{1-\beta^{2m}} + \frac{32\alpha^2 m^3\ccD_2^2}{n\beta^2_{0:m}(1-\beta^{2m})}}} \crk{4\C \brk{\psi\prt{\proxp{\bar{z}_{k_1}}} - \bpsi} + 2\C \sigfn + \sigma^2}.
        \end{aligned}
    \end{equation}

    Substituting \eqref{eq:cH_ek_d} and \eqref{eq:Euk_ub_d} into \eqref{eq:cLk1k2} leads to
    \begin{equation}
        \label{eq:cL_d}
        \begin{aligned}
            &\condE{\cL_{k_2}}{\cF_{k_1}} \leq \brk{1 + 12\alpha \C H_1 + \frac{400\alpha^3 \beta^2_{0:m} \C \ccD_1^2}{1-\beta^2}\prt{H_2 + \frac{32\alpha^2m^2\ccD_2^2 \norm{D_1}^2}{1-\beta^{2m}} + \frac{96\alpha^2 m^3\ccD_2^2}{n\beta^2_{0:m}(1-\beta^{2m})}}}\cL_{k_1} \\
            &\quad + \brk{\frac{24\alpha \ccD_1^2\beta_{0:m}^2}{n} + \prt{\frac{1 + \beta^{2m}}{2} + \frac{64\alpha^2m \ccD_2^2 }{1-\beta^{2}}} \frac{50\alpha\ccD_1^2}{n\prti{1-\beta^2}}}\norm{\bu_{k_1}}^2\\
            &\quad - \alpha m\brk{\frac{\cC_0}{2} - \frac{6\C\gamma}{mn} - 6\alpha^2m^2\ccD_1^2\prt{\norm{D_1}^2 + 2} - \frac{200\alpha^2 \ccD_1^2 \norm{D_1}^2 \beta^2_{0:m}}{m(1-\beta^2)} - \frac{800 \alpha^4\Ln^2\ccD_1^2\norm{D_2}^2}{(1-\beta^2)^2} \right.\\
            &\quad \left.- \frac{800\alpha^4 m^3\ccD_1^2\ccD_2^2\prt{\norm{D_1}^2 + 2}}{(1-\beta^2)(1-\beta^{2m})}}\norm{\Fnor{\bar{z}_{k_1}}}^2\\
            &\quad + \brk{3\alpha H_1 + \frac{100\alpha^3\beta_{0:m}^2\ccD_1^2}{1-\beta^2}\prt{H_2 + \frac{32\alpha^2m^2\ccD_2^2 \norm{D_1}^2}{1-\beta^{2m}} + \frac{96\alpha^2 m^3\ccD_2^2}{n\beta^2_{0:m}(1-\beta^{2m})}}} \prt{2\C\sigfn + \sigma^2}\\
        \end{aligned}
    \end{equation}

    To obtain the desired result in \eqref{eq:cL}, we first simplify the coefficient of $\normi{\bu_{k_1}}^2$:
    \begin{align*}
        \tc_3:=\frac{24\alpha \ccD_1^2\beta_{0:m}^2}{n} + \prt{\frac{1 + \beta^{2m}}{2} + \frac{64\alpha^2m \ccD_2^2 }{1-\beta^{2}}} \frac{50\alpha\ccD_1^2}{n\prti{1-\beta^2}}.
    \end{align*}
    Letting $\alpha\leq \sqrt{(1-\beta^2)(1-\beta^{2m})}/(40\sqrt{2m} \ccD_2)$ leads to $\tc_3\leq 50\alpha \ccD_1^2/[n(1-\beta^2)]$.

    We next simplify the coefficient of $\normi{\Fnor{\bar{z}_{k_1}}}^2$ in \eqref{eq:cLk1k2} below. 
    \begin{align*}
       &\tc_4:= - \alpha m\brk{\frac{\cC_0}{2} - \frac{6\C\gamma}{mn} - 6\alpha^2m^2\ccD_1^2\prt{\norm{D_1}^2 + 2}  - \frac{200\alpha^2 \ccD_1^2 \norm{D_1}^2 \beta^2_{0:m}}{m(1-\beta^2)} - \frac{800 \alpha^4\Ln^2\ccD_1^2\norm{D_2}^2}{(1-\beta^2)^2}\right.\\
       &\left. - \frac{800\alpha^4 m^3\ccD_1^2\ccD_2^2\prt{\norm{D_1}^2 + 2}}{(1-\beta^2)(1-\beta^{2m})}}\leq -\frac{\alpha m}{9},
    \end{align*}
    where we let invoked conditions on $\alpha$ and $\gamma$:
    \begin{align*}
        \alpha&\leq\min\crk{\frac{1}{80m\ccD_1\sqrt{\normi{D_1}^2 + 2}}, \frac{1-\beta^2}{80\ccD_1\normi{D_1}}, \frac{\sqrt{1-\beta^2}}{80\normi{D_2}\ccD_1}, \frac{\sqrt{1-\beta^2}}{80\Ln^2},\frac{\sqrt{(1-\beta^2)(1-\beta^{2m})}}{80\sqrt{m}\ccD_2}},\\
        \gamma&\leq \min\crk{\frac{1}{18\C}, \frac{1}{5(L + \rho)}}.
    \end{align*}
    
    The coefficient of $(2\C\sigfn + \sigma^2)$ is simplified as follows:
    \begin{align*}
        \tc_5&:=3\alpha H_1 + \frac{100\alpha^3\beta_{0:m}^2\ccD_1^2}{1-\beta^2}\prt{H_2 + \frac{32\alpha^2m^2\ccD_2^2 \norm{D_1}^2}{1-\beta^{2m}} + \frac{96\alpha^2 m^3\ccD_2^2}{n\beta^2_{0:m}(1-\beta^{2m})}}\\
        &= \frac{3\alpha}{n}\brk{1 + 8\alpha^2m^2\ccD_1^2  + \frac{3200\alpha^4 m^3\ccD_1^2 \ccD_2^2}{(1-\beta^2)(1-\beta^{2m})}} + 4\alpha^3\norm{D_1}^2\ccD_1^2  \brk{6m\beta_{0:m}^2 + \frac{25\beta_{0:m}^2}{1-\beta^2} + \frac{800\alpha^2 \beta_{0:m}^2\ccD_2^2}{(1-\beta^2)(1-\beta^{2m})}} \\
        &\quad + \frac{400\alpha^5 m\beta_{0:m}^2\ccD_1^2\normi{D_2}^2\Ln^2}{n(1-\beta^2)(1-\beta^{2m})}\\
        &\leq \frac{5\alpha}{n} + \frac{133\alpha^3 m\normi{D_1}^2\ccD_1^2}{1-\beta^2} + \frac{400\alpha^5 m\beta_{0:m}^2\ccD_1^2\normi{D_2}^2\Ln^2}{n(1-\beta^2)(1-\beta^{2m})},
    \end{align*}
    where we invoked $\beta_{0:m}^2(1-\beta^2) = 1-\beta^{2m}$ and 
    \begin{align*}
        \alpha&\leq\min\crk{\frac{1}{80m\ccD_1},\frac{\sqrt{(1-\beta^2)(1-\beta^{2m})}}{80\sqrt{m}\ccD_2}}
    \end{align*}
    in the last inequality.

    Noting the above simplified coefficients yields the desired result \eqref{eq:cL}.

\subsection{Proof of Lemma \ref{lem:sum_dist}}
    \label{app:sum_dist}
    
    According to Lemma \ref{lem:Fnor_Lip}, we have 
    \begin{equation}
        \label{eq:Fnor_inter}
        \norm{\Fnor{\bar{z}_{\tk_t + q}}}^2 \leq 2\Ln^2 \norm{\bar{z}_{\tk_t} - \bar{z}_{\tk_t + q}}^2 + 2\norm{\Fnor{\bar{z}_{\tk_t}}}^2,\; \forall t\geq 0,\; 0\leq q\leq m-1.
    \end{equation}
    
    Substituting \eqref{eq:Fnor_inter} into \eqref{eq:sum_norm} leads to
    \begin{equation}
        \label{eq:sum_norm_app}
        \begin{aligned}
            \sum_{k=0}^{K-1} \E\brk{\norm{\Fnor{\bar{z}_k}}^2}
            &\leq 2\Ln^2 \sum_{t=0}^T \sum_{q=0}^{m-1}\E\brk{\norm{\bar{z}_{\tk_t} - \bar{z}_{\tk_t + q}}^2} + 2m \sum_{t=0}^T \E\brk{\norm{\Fnor{\bar{z}_{\tk_t}}}^2}.
        \end{aligned}
    \end{equation}
    
    The consensus error can be bounded similarly as follows:
    
    \begin{equation}
        \label{eq:sum_cons}
        \begin{aligned}
            \sum_{k=0}^{K-1}\frac{1}{n}\norm{\Pi\z_k}^2 \leq \sum_{t=0}^T \sum_{q=0}^{m-1}\frac{1}{n}\norm{\Pi\z_{\tk_t + q}}^2\leq \sum_{t=0}^T \sum_{q=0}^{m-1}\frac{\norm{V}^2}{n}\norm{\bu_{\tk_t + q}}^2.
        \end{aligned}
    \end{equation}
    
    Thus, substituting \eqref{eq:sum_norm_app} and \eqref{eq:sum_cons} into \eqref{eq:Fnat2Fnor} leads to the desired result. 

\end{document}